\documentclass[11pt,a4paper]{amsart}

\usepackage{graphicx}        
\usepackage{multicol}        
\usepackage[bottom]{footmisc}
\usepackage{xcolor}

\usepackage{newtxtext}       %
\usepackage[varvw]{newtxmath}       

\usepackage{amsmath}
\usepackage{mathtools}
\usepackage{physics}

\theoremstyle{definition}
\newtheorem{theorem}{Theorem}[section]

\newtheorem{proposition}[theorem]{Proposition}

\newtheorem{lemma}[theorem]{Lemma}
\newtheorem{definition}[theorem]{Definition}

\numberwithin{figure}{section}
\numberwithin{table}{section}
\numberwithin{equation}{section}

\renewcommand{\phi}{\varphi}
\renewcommand{\epsilon}{\varepsilon}

\newcommand{\R}{\mathbb{R}}
\newcommand{\Qad}{Q^{\text{ad}}}

\let\bs\boldsymbol
\newcommand\boldu{\bs u}

\newcommand\boldv{\bs v}

\newcommand{\fracture}[1]{#1}

\begin{document}
\date{\today}

\title{Coefficient Control of Variational Inequalities}

\author[Hehl]{Andreas Hehl}
\address{Andreas Hehl: Universität Bonn, Institut für Numerische
  Simulation
  }
\email{hehl@ins.uni-bonn.de}
\author[Khimin]{Denis Khimin}
\address{Denis Khimin: Leibniz Universität Hannover, Fakultät für
  Mathematik und Physik, Institut für Angewandte Mathematik
  }
\email{khimin@ifam.uni-hannover.de}
\author[Neitzel]{Ira Neitzel}
\address{Ira Neitzel: Universität Bonn, Institut für Numerische
  Simulation
  }
\email{neitzel@ins.uni-bonn.de}
\author[Simon]{Nicolai Simon}
\address{Nicolai Simon: Universit{\"a}t Hamburg, MIN Fakult{\"a}t,
  Fachbereich Mathematik
  }
\email{nicolai.simon@uni-hamburg.de}
\author[Wick]{Thomas Wick}
\address{Thomas Wick: Leibniz Universität Hannover, Fakultät für
  Mathematik und Physik, Institut für Angewandte Mathematik
  }
\email{thomas.wick@ifam.uni-hannover.de}
\author[Wollner]{W.~Wollner}
\address{Winnifried Wollner: Universit{\"a}t Hamburg, MIN Fakult{\"a}t,
  Fachbereich Mathematik
   }
\email{winnifried.wollner@uni-hamburg.de}

\maketitle

\begin{abstract}
  Within this chapter, we discuss control in the coefficients of an
obstacle problem. Utilizing tools from H-convergence, we show
existence of optimal solutions. First order necessary optimality
conditions are obtained after deriving directional differentiability
of the coefficient to solution mapping for the obstacle problem.
Further, considering a regularized obstacle problem as a constraint yields a limiting optimality system after proving, strong, convergence of
the regularized control and state variables. Numerical examples underline convergence with respect to the regularization.
\fracture{Finally, some numerical experiments highlight the possible
  extension of the results to coefficient control in phase-field fracture.}

\end{abstract}


\section{Introduction}
In this chapter, we consider an optimization problem of the form 
\begin{equation}\label{eq:P}
  \begin{aligned}
    \min &\;J(q,u) = j(u)+\frac{\alpha}{2}\|q\|^2\\
    &\text{s.t.}\;\left\{
      \begin{aligned}
        (q \nabla u, \nabla (v-u)) &\ge (f,v-u) &  \forall v &\in K,\\
        u & \in K, &  q &\in \Qad,
      \end{aligned}
    \right.
  \end{aligned}
\end{equation}
governed by an obstacle problem in a domain $\Omega\subset \R^d$, 
where $d$ denotes the spatial dimension, 
for a control $q$ in an admissible set $\Qad \subset
L^2(\Omega;\R^{d\times d}_{\text{sym}})$ that is closed and convex, and a state $u\in K \subset H^1_0(\Omega)$. 
The precise mathematical problem statement
will be presented in Section~\ref{sec:existence}. In the following, let us briefly 
comment on the control-to-state coupling.
For each given fixed, uniformly positive definite, control $q\in \Qad$ acting as a coefficient function and data $f$ the obstacle problem of finding $u \in K$ solving
\begin{equation*}
  (q \nabla u, \nabla (v-u)) \ge (f,v-u) \quad  \forall v \in K,
\end{equation*}
has a well developed
theory providing existence and regularity of solutions $u$, see,
e.g.,~\cite{KinderlehrerStampacchia:2000,Rodrigues:1987}.

Optimization problems similar to~\eqref{eq:P}, but with control acting in the right hand side $f$,
rather than the coefficient $q$, have been investigated over 
many years. Indeed, even in this case the obstacle problem gives rise to
a non-differentiable operator $f \mapsto u$, in general.
Early works by~\cite{Haraux:1977} provided directional
differentiability, and~\cite{Mignot:1976,MignotPuel:1984} provide
necessary optimality conditions for such problems. Similar results for
constraints of Signorini rather than obstacle type can 
be found in~\cite{Barbu:1981}.
For an overview of these results;
see also~\cite{Barbu:1984} or~\cite{BonnansShapiro:2000}.

The inherent non-differentiability of the previous problem statement, with control in the right hand side $f,$
motivated the investigation of relaxation approaches for the
variational inequality in~\cite{Bergounioux:1997}. A scheme
that allows for an efficient solution is the primal-dual
active set method proposed in~\cite{ItoKunisch:2000}. A convergence analysis
for a similar regularization approach was established in~\cite{SchielaWachsmuth:2013}.

The lack of differentiability results in the difficulty of
asserting suitable necessary optimality conditions for this problem
class and different stationarity concepts, such as strong, weak, C-, or
M-stationarity need to be considered. Indeed, strong stationarity is a
necessary optimality condition if suitable compatibility conditions on
the control bound are satisfied and the control space is large
enough~\cite{Wachsmuth:2014},
but in more general situations weaker concepts need to be considered,
see, e.g.~\cite{Wachsmuth:2016}. We refer also
to~\cite{HarderWachsmuth:2018} for a comparison of different
stationarity concepts. Recently,~\cite{RaulsUlbrich:2019}
characterized the Bouligand generalized differential for the mapping
$f \mapsto u$ given by the obstacle problem, and~\cite{Christof:2019}
obtained stationarity conditions for time dependent variational
inequalities of obstacle type. Moreover, the authors 
of~\cite{AlphonseHintermuellerRautenberg:2019} provided
directional differentiability results for quasi-variational
inequalities of obstacle type with control in the right hand side $f$.
Further,~\cite{ChristofWachsmuth:2020} established sensitivity results
for variational inequalities of second kind.

Algorithmic approaches for this problem class can be based on
the regularization of the variational
inequality~\cite{KunischWachsmuth:2012} coupled with a path-following
strategy~\cite{KunischWachsmuth:2012b}. The latter can also be coupled with
adaptive mesh refinement utilizing a posteriori 
error estimates~\cite{MeyerRademacherWollner:2014}.
Alternatively, non-smooth optimization techniques such as
bundle-methods can be combined with inexact solutions of the
sub-problems as proposed in~\cite{HertleinUlbrich:2019}.
Based on the observation that real valued Lipschitz functions on
Banach spaces are differentiable on a dense subset, if the norm is
differentiable away from zero~\cite{Preiss:1990}, it was proposed
by~\cite{ChristofReyesMeyer:2017} to utilize smooth sub-problems in a
trust-region framework whenever differentiability can be asserted and
only fall back to the use of non-linear directional derivatives if
differentiability fails.

The novelty of this book chapter is the extension of some of the previously mentioned
findings to
control in the coefficient $q$ rather than the right hand side $f$.
In this context, existence no longer follows from the compactness of
the mapping $L^2 \ni f \mapsto u \in H^1$, and
H-convergence~\cite{MuratTartar:1985,MuratTartar:1997,Tartar:1997} or
G-convergence~\cite{Spangnolo:1976} needs to be considered to analyze
well-posedness of the problem. We also refer the reader to~\cite{Allaire:2002}. These techniques have been utilized successfully in
the context of free-material
optimization~\cite{HaslingerKocvaraLeugeringStingl:2010} as well as in
the discretization error analysis of matrix identification problems~\cite{DeckelnickHinze:2011a,DeckelnickHinze:2012}.

\fracture{Furthermore, we provide a proof of concept to extend the 
  coefficient control to phase-field fracture. 
  The variational approach to fracture, known as phase-field nowadays, 
  goes back to~\cite{FraMar98,BourFraMar00,KuMue10,MieWelHof10a}.
  Monographs, recent overviews, and two phase-field benchmark settings
  are provided
  in~\cite{BourFraMar08,AmGeraLoren15,BouFra19,Wi20_book,Schroeetal20,Fra21,DiLiWiTy22}. 
  Our formulation starts from our own prior work~\cite{NeitzelWickWollner:2017,NeiWiWo19}.
  In the current work, the coefficient control acts in the elasticity 
  tensor as it is common in free-material optimization, see,
  e.g.,~\cite{HaslingerKocvaraLeugeringStingl:2010}. A prototype setting 
  is described and investigated in a computational fashion.
}

The rest of the manuscript is organized as follows. In
Section~\ref{sec:existence}, we will state the precise problem under
consideration and collect some well known facts on the obstacle
problem needed in the upcoming analysis. We will then discuss the
well-posedness of our optimization problem.
In Section~\ref{sec:differentiability}, we will discuss directional
differentiability of the control to state map, generalizing results
of~\cite{Mignot:1976} to the case of control in the coefficients, and eventually obtain a first formulation of first order necessary optimality conditions. These results will be
detailed in a forthcoming publication. In Section~\ref{sec:optcond_regular},
we discuss recent results concerning optimality conditions for a regularized variational inequality
as given in~\cite{SimonWollner:2023}. In
Section~\ref{sec:numerics}, we show some numerical results for the
coefficient control in the obstacle problem highlighting the
convergence of the regularized solutions to the limiting VI-solution.
\fracture{Finally, in Section~\ref{sec:fracture}, we provide a
  prototypical extension towards phase-field fracture and provide some
  illustrating numerical examples.
}
The article closes with an outlook and summary of further project
results in Section~\ref{sec:outlook}.
\section{Problem Statement and Existence of Solutions}\label{sec:existence}
In this section, we start by collecting precise assumptions for our model problem and eventually show existence of at least one global minimizer.
\subsection{Notation and Assumptions}
Let us first agree on some general notation and underlying assumptions.
Let $\Omega \subset \R^d$ be a given Gr\"oger-regular domain, c.f. \cite{Groeger:1989}, where $d \in
\{ 1,\ldots,3 \}$. We use the notation $Q:= L^2(\Omega;\R^{d\times
  d}_{\text{sym}})$ and $U:=H_0^1(\Omega)$, $\|\cdot\|$ and
$(\cdot,\cdot)$ denote the $L^2(\Omega)$-norm and inner product for
scalar, vector, and matrix valued functions. The admissible sets for the control and state are defined as
\begin{align*}
  \Qad &= \{ q \in L^2(\Omega;\R^{d\times d}_{\text{sym}})\,|\, 0
  \prec q_{\min}I \preccurlyeq q(x) \preccurlyeq q_{\max}I,\; \text{a.e.} \}
  \subset Q,\\
  K &= \{ u \in H^1_0(\Omega)\,|\, u \le \psi\} \subset U, 
\end{align*}
for given $q_{\min},q_{\max}, \psi \in \R$ with $q_{\min}<q_{\max}$, $\psi > 0$, and $I$ denoting the identity matrix. Here
$\prec$ and $\preccurlyeq$ denote the standard ordering of symmetric
matrices given by the positive definite and positive semi-definite cone, respectively.

Moreover, let $f\in L^2(\Omega)$ be fixed, $\alpha>0$ a positive cost parameter, and $j \colon U
\rightarrow \R$ a weakly lower semicontinuous Fr{\'e}chet-differentiable functional that is bounded from below. Then, we will consider the optimization problem
\begin{equation}\label{eq:opt}\tag{P}
  \begin{aligned}
    \min &\;J(q,u) = j(u)+\frac{\alpha}{2}\|q\|^2\\
    &\text{s.t.}\;\left\{
      \begin{aligned}
        (q \nabla u, \nabla (v-u)) &\ge (f,v-u) &  \forall v &\in K,\\
        u & \in K, &  q &\in \Qad.
      \end{aligned}
    \right.
  \end{aligned}
\end{equation}

Within this setup it is known that for any given $q \in \Qad$, the obstacle problem of finding $u \in
K$ solving
\begin{equation}\label{eq:obstacle}
  (q \nabla u, \nabla (v-u)) \ge (f,v-u) \quad  \forall v \in K
\end{equation}
admits a unique solution, which follows immediately from the 
equivalent strictly convex minimization problem
\begin{equation*}
  u = \operatorname{arg min}_{v \in K} \frac{1}{2} (q\nabla v,\nabla
  v) - (f,v).
\end{equation*}
This in turn is equivalent to the existence of a Lagrange multiplier $\lambda\in H^{-1}(\Omega)$ such that
\begin{equation*}
  (q\nabla u,\nabla \varphi)=(f,\varphi)-\langle \lambda,\varphi\rangle\quad\forall\varphi \in H_0^1(\Omega).
\end{equation*}
Since $f \in L^2(\Omega)$ and $\psi \in
\R$,~\cite[Chapter~5, Proposition~2.2]{Rodrigues:1987} shows that in fact
the solution satisfies $\nabla \cdot (q\nabla u) \in L^2(\Omega)$ and
thus we can define the multiplier
\begin{equation*}
  \lambda := \nabla \cdot (q\nabla u) + f \in L^2(\Omega).
\end{equation*}
Here it should be noted that $L^\infty$-regularity of the coefficient $q\in\Qad$ is in general not sufficient for $u  \in
H^2(\Omega)$.

Following~\cite[Chapter~5, Proposition~2.2]{Rodrigues:1987}, and using
that $\psi$ is constant, we find
\begin{equation}\label{eq:lambda-bound}
  \begin{aligned}
    \| \nabla \cdot (q\nabla u)\| &\le 2\|f\|,\\
    \|\lambda\| &= \|\nabla \cdot (q\nabla u)+f\|\\
    &\le 3 \|f\|.
  \end{aligned}
\end{equation}
From this, the Lax-Milgram theorem asserts the uniform bound
\begin{equation}\label{eq:u-bound}
  \|\nabla u\| \le c\frac{q_{\max}}{q_{\min}}\|f\|, 
\end{equation}
and eventually the solution $u$ of~\eqref{eq:obstacle} can
equivalently be characterized by the
complementarity system
\begin{equation}\label{eq:obstacle-comp}
  \begin{aligned}
    -\nabla \cdot (q\nabla u) + \lambda &= f & &\text{in } L^2(\Omega),\\
    \lambda & \ge 0& & \text{in } L^2(\Omega),\\
    u & \le \psi, & & \text{q.e. in }\Omega,\\
    (\lambda,u-\psi) &= 0.
  \end{aligned}
\end{equation}

For the remainder of this chapter, we define the control-to-state mapping $S\colon \Qad\to U$, $q\mapsto u$, where $u$ solves the obstacle problem~\eqref{eq:obstacle} for a given coefficient function $q\in\Qad$.

\subsection{Existence of Solutions}
In this subsection, we discuss existence of solutions
to~\eqref{eq:opt}. Due to the appearance of the product $q \nabla u$
in the variational inequality (VI)~\eqref{eq:obstacle},
weak$^*$ convergence of the control as it is induced by the
boundedness of $\Qad \subset L^\infty(\Omega;\R^{d\times d}_{\text{sym}})$ is not sufficient for
passage to the limit in the VI. To circumvent this difficulty we
resort to $H$-convergence, see, e.g.,~\cite[Section~2]{Tartar:1997}.
\begin{definition}
  A sequence $q_k \in \Qad$ $H$-converges to $q\in \Qad$ (in symbols
  $q_k \xrightarrow{H}q$) if for any sequence $u_k \in U$ 
  satisfying
  \begin{align*}
    u_k &\rightharpoonup u,& \text{in }&U,\\
    \nabla \cdot (q_k\nabla u_k)&\rightarrow f,&\text{in }&U^*
  \end{align*}
  for some $u \in U$ and $f \in U^*$ it holds 
  \begin{equation*}
    q_k\nabla u_k \rightharpoonup q \nabla u, \quad \text{in }L^2(\Omega;\R^d).
  \end{equation*}
\end{definition}
With the help of this concept we obtain the following existence result:
\begin{theorem}\label{thm:existence}
  There exists at least one solution of~\eqref{eq:opt}.
\end{theorem}
\begin{proof}
  The proof follows the standard line of arguments. Clearly, $J$ is
  bounded from below, and thus we can select a minimizing sequence
  $(q_k,u_k) \in \Qad \times K$, with corresponding Lagrange multiplier
  $\lambda_k$. 
  Due to $H$-compactness of $\Qad$, see,
  e.g.,~\cite[Theorem~1.2.16]{Allaire:2002},
  we can select an $H$-convergent subsequence $q_k$ with limit $q$.
  Further, by the bound~\eqref{eq:u-bound} $u_k \rightharpoonup u$ in
  $U$ and by~\eqref{eq:lambda-bound}
  \begin{equation*}
    \nabla \cdot (q_k\nabla u_k)\rightharpoonup g \quad\text{in } L^2(\Omega)
  \end{equation*}
  for some $g \in L^2(\Omega)$.
  By compactness of the embedding $L^2(\Omega) \subset U^*$ the convergence is strong in $U^*$
  and the definition of $H$-convergence asserts
  \begin{equation*}
    q_k\nabla u_k \rightharpoonup q \nabla u,\quad \text{in }L^2(\Omega;\R^d).
  \end{equation*}
  Further, the bound~\eqref{eq:lambda-bound}
  implies weak convergence $\lambda_k \rightharpoonup \lambda$ in
  $L^2(\Omega)$. 
  Combined, we see that the limit $(q,u,\lambda)$ satisfies the first equation
  in~\eqref{eq:obstacle-comp}. Clearly, $K$ is closed with respect to
  weak convergence in $U$, i.e., $u \le \psi$ holds, and the
  complementarity relation $(\lambda,\psi-u) =0$ follows from strong
  convergence of $u_k$ in $L^2(\Omega)$. The sign condition on $\lambda$
  follows immediately by Mazur's lemma. Thus, the limit satisfies
  the system~\eqref{eq:obstacle-comp}, which means that the pair $(q,u)$
  solves~\eqref{eq:obstacle}.
  
  It remains to see the lower-semicontinuity of $J$. The first term is
  weakly lower semicontinuous by assumption, while the norm on $q$ is
  lower-semicontinuous with respect to $H$-convergence
  by~\cite[Lemma~2.1]{DeckelnickHinze:2011a}.
  This shows the assertion.
\end{proof}

\section{First Order Necessary Optimality Conditions}\label{sec:differentiability}
In this section, we first prove directional differentiability of the control-to-state mapping $S$ and eventually derive a first order optimality condition in form a variational inequality, that can be viewed as a first step towards optimality conditions in qualified form.
\subsection{Directional Differentiability of $S$}In order to prove directional differentiability of the control to state map
$S$ in the sense of Hadamard, see,
e.g.,~\cite[Definition~2.2]{Shapiro:1990}, we need to show that for a given control $q \in \Qad$ and direction $d \in Q$
with $q+d \in \Qad$, the limit
\begin{equation*}
  S'(q;d) := \lim_{t \downarrow 0} \frac{S(q+td)-S(q)}{t}.
\end{equation*}
exists. Let us proceed in two steps. First, let us define the auxiliary operator
\begin{equation*}
  \hat S_q\colon H^{-1}(\Omega)\to H_0^1(\Omega),\quad f\mapsto u,
\end{equation*}
with $u\in K$ such that
\begin{equation*}
  (q\nabla u,\nabla (\varphi-u))\ge (f,\varphi-u),\quad \forall \varphi\in K.
\end{equation*}
In other words, in this operator the coefficient matrix $q\in\Qad$ is
fixed, and the right-hand-side $f\in H^{-1}(\Omega)$ is mapped to the
solution of the classical obstacle problem with the given coefficient $q$. Note that then we have  \begin{equation*}u=S(q)=\hat S_q (f).\end{equation*}
As already pointed out in the introduction, the operator $S_q$ acting
on a control in the right hand side is well understood.  Applying the
results of~\cite{Mignot:1976} and~\cite{Wachsmuth:2014}, we obtain:

\begin{lemma}
  \label{thm:diffSq}
  The operator $\hat S_q\colon H^{-1}(\Omega)\to H_0^1(\Omega)$ is directionally differentiable at all points $f\in H^{-1}(\Omega)$. Its directional derivative $z:=\hat S_q'(f;h)\in H_0^1(\Omega)$ in direction $h\in H^{-1}(\Omega)$ is given by the unique solution of the variational inequality
  \begin{equation}
    z\in\mathcal{K}(f),\quad (q\nabla z,\nabla(\varphi-z))\ge (h,\varphi-z)\quad \forall \varphi\in \mathcal{K}(f),
  \end{equation}
  with 
  \begin{equation*}
    \mathcal{K}(f)=\{\varphi\in H_0^1(\Omega)\colon \varphi\le 0\text{ q.e. on }\{\hat S_q(f)=\psi\} \text{ and } \varphi=0 \text{ q.e. on q-supp}(\lambda) \}, 
  \end{equation*}
  where $u=\hat S_q (f)$ and $\lambda=\nabla\cdot(q\nabla u)+f$ is the associated Lagrange multiplier defined in~\eqref{eq:obstacle-comp}.
\end{lemma}
\begin{proof}This is a formulation
  of~\cite[Theorem~2.9]{NeitzelWachsmuth}, cf.,~\cite{Mignot:1976},
  applied to $\hat S_q$ and adopted to our notation. The concrete form of the critical cone is due to~\cite{Wachsmuth:2014}. \end{proof}
Now, as pointed out in~\cite{NeitzelWachsmuth}, the operator $\hat S_q$ is Lipschitz continuous and therefore even Hadamard differentiable. This yields our desired differentiability result for the operator $S$:
\begin{theorem}
  \label{thm:diffS}
  The operator $S\colon Q\to H_0^1(\Omega)$ is directionally differentiable at all points $q\in \Qad$. Its directional derivative $\tilde u:=S'(q;d)\in H_0^1(\Omega)$ in direction $d\in Q$ with $q+d\in \Qad$ is given by the unique solution of the variational inequality
  \begin{equation}\label{eq:sprime}
    \tilde  u\in\mathcal{K}(q),\quad (q\nabla \tilde u,\nabla(\varphi-\tilde u))\ge (\nabla \cdot(d\nabla u),\varphi-\tilde u)\quad \forall \varphi\in \mathcal{K}(q),
  \end{equation}
  with $u=S(q)$ and 
  \begin{equation*}
    \mathcal{K}(q)=\{\varphi\in H_0^1(\Omega)\colon \varphi\le 0\text{ q.e. on }\{ u=\psi\} \text{ and } \varphi=0 \text{ q.e. on q-supp}(\lambda) \}, 
  \end{equation*}
\end{theorem}
\begin{proof}
  Set $u=S(q)$ and note that $u_{td}=S(q+td)$ fulfills the variational inequality
  \begin{equation}\label{eq:aux1}
    u_{td}\in K\quad ((q+td)\nabla u_{td},\nabla (\varphi-u_{td}))\ge (f,\varphi-u_{td})\quad \forall\varphi\in K,
  \end{equation}
  which is equivalent to
  \begin{equation*}
    u_{td}\in K\quad (q\nabla u_{td},\nabla (\varphi-u_{td}))\ge (f,\varphi-u_{td})-t(d\nabla u_{td},\nabla(\varphi-u_{td}))\quad \forall\varphi\in K.
  \end{equation*}
  Using integration by parts in the right-hand-side of the last inequality we deduce \begin{equation*}S(q+td)=u_{td}=\hat S_q(f+\nabla \cdot(td\nabla u_{td})),\end{equation*} with
  \begin{equation*}f+\nabla\cdot(td\nabla u_{td})\in H^{-1}(\Omega). \end{equation*}
  We note that~\eqref{eq:aux1} implies that $\nabla\cdot(d\nabla
  u_{td})$ tends to $\nabla \cdot (d\nabla u)$ in $H^{-1}$ as
  $t$ tends to zero. Therefore, applying the
  Hadamard-differentiability of $\hat S_q$, i.e., Lemma~\ref{thm:diffSq}, we observe
  \begin{equation*}
    \frac{(S(q+td)-S(q))}{t}=\frac{\hat S_q(f+t\nabla \cdot(d\nabla u_{td})-\hat S_q(f)}{t}\to \hat S_q'(f;\nabla\cdot(d\nabla u))
  \end{equation*}
  as $t\to 0$, and hence $S$ is directionally differentiable with $S'(q;d)=\hat S_q'(f;\nabla \cdot(d\nabla u))$.
  
\end{proof}

Following~\cite{Wachsmuth:2014}, we obtain an analogue to the complementarity system~\eqref{eq:obstacle-comp}, namely that the variational inequality~\eqref{eq:sprime}
is equivalent to the complementarity system
\begin{equation}\label{eq:obstaclederivative-comp}
  \begin{aligned}
    -\nabla \cdot (q\nabla \tilde u) + \tilde\lambda &=\nabla\cdot(d\nabla u)  &\text{in } H^{-1}(\Omega),\\
    \tilde \lambda & \in \mathcal{K}(\bar q)^\circ,\\
    \tilde u & \in \mathcal{K}(\bar q),\\
    \langle\tilde \lambda,\tilde u\rangle_{H^{-1}(\Omega), H_0^1(\Omega)} &= 0.
  \end{aligned}
\end{equation}

\subsection{Optimality Conditions}\label{sec:optcond}
To obtain necessary optimality conditions, we rewrite our optimization problem~\eqref{eq:opt} into the, usual, reduced form utilizing the control-to-state mapping $S$:

\begin{equation*}
  \min\limits_{q\in\Qad} \;f(q): = j(S(q))+\frac{\alpha}{2}\|q\|^2.
\end{equation*}
Note that the Fr{\'echet} differentiability of $j$ and $\|\cdot\|^2$ and the directional differentiability of $S$ thanks to Theorem~\ref{thm:diffS} yields directional differentiability of $f$. The following primal first order necessary condition is then a straight forward consequence of the convexity of $\Qad$.
\begin{lemma}\label{lem:fonvi}
  Let $\bar q\in \Qad$ be a locally optimal control for~\eqref{eq:opt} with associated state $\bar u=S\bar q$. Then the following variational inequality is fulfilled:
  \begin{equation*}
    (j'(\bar u), S'(\bar q;q-\bar q))+\alpha(\bar q,q-\bar q)\ge 0\quad \forall q\in \Qad.
  \end{equation*}
\end{lemma}
\begin{proof}
  The proof follows by standard arguments. We refer to,
  e.g.,~\cite[Lemma 3.2]{NeitzelWachsmuth} for a similar setting
  with control in the right-hand-side
  and additional state constraints. Since $\bar q\in \Qad$ and
  $\Qad$ is convex,
  we know that $q_t:=\bar q+t(q-\bar q)\in \Qad$ for all
  $q\in\Qad$.
  Since $\bar q$ is locally optimal, we observe
  \begin{align*}
    0&\le f(q_t)-f(\bar q)\\
    &= j(S(q_t))-j(S(\bar q)) +\frac{\alpha}{2}
    (\|q_t\|^2-\|\bar q\|^2)
  \end{align*}
  for all $q\in \Qad$ and $t$ sufficiently small. Since both $j$
  and $\|\cdot\|^2$ are Fr{\'e}chet differentiable, dividing by
  $t>0$ and passing to the limit yields
  \begin{equation*}
    0\le  j'(\bar u)(S'(\bar q; q-\bar q)+\alpha(\bar q,q-\bar q)\quad\forall q\in \Qad
  \end{equation*}
  by Theorem~\ref{thm:diffS}.
\end{proof}

\section{Optimality Conditions for a Regularized Problem}\label{sec:optcond_regular}

In this section, we introduce a set of limiting optimality conditions for~\eqref{eq:opt} on domains $\Omega \subset \mathbb{R}^2$ by utilizing a regularization approach and considering the limit
points of stationarity conditions for a series of regularized problems, similar to the approach in, e.g.,~\cite{MeyerRademacherWollner:2014}. To this effect, we introduce a
regularized version of the obstacle problem and consider the limits of its optimal solutions. Additional supporting results regarding the regularity estimates used can be found in~\cite{SimonWollner:2023} and a more detailed explanation of the results presented in this section with all pertinent proofs will be provided in a forthcoming paper.

For the regularized problem, we introduce a biquadratic penalization
$r: \mathbb{R}^{+} \times \mathbb{R} \to \mathbb{R}$ of the obstacle
energy functional into the problem, see,~\cite{MeyerRademacherWollner:2014}.
The resulting problem is given by 
\begin{equation}\label{eq:Regularized-Obstacle-Problem}\tag{$P_{\gamma}$}
  \begin{aligned}
    \min_{q_\gamma , u_\gamma} \quad &J\left( q_\gamma ,u_\gamma \right) = \frac{1}{2} \| u_\gamma - u_d \|^2_2 + \frac{\alpha}{2} \left\| q_\gamma \right\|^2_2\\
    \text{s.t. } &- \nabla \cdot \left( q_\gamma \nabla u_\gamma \right) + r\left( \gamma ; u_\gamma \right) = f \quad \text{in }H^{-1}(\Omega), \\
    &u_\gamma \in H_0^1(\Omega), \quad q_\gamma \in Q^{\rm{ad}},
  \end{aligned}
\end{equation}
with
\begin{equation*}
  \begin{aligned}
    r\left( \gamma ; u_\gamma \right) :=  \gamma \left[ \max \left(
        \left(  u_\gamma - \psi \right), 0 \right)^3 \right].
  \end{aligned}
\end{equation*}

Similar to the penalizations in, e.g.,~\cite{MeyerRademacherWollner:2014, SchielaWachsmuth:2013}, the penalization $r \left( \gamma ; u_\gamma \right)$ describes a locally Lipschitz continuous, monotone Nemyzkii operator. Also note, that the control is a positive definite and symmetric operator. Since, given a control $q_\gamma \in Q^{\rm{ad}}$, the left-hand side
of the PDE
\begin{equation*}
  \begin{aligned}
    - \nabla \cdot \left( q_\gamma \nabla u_\gamma \right) + r\left( \gamma ; u_\gamma \right) = f
  \end{aligned}
\end{equation*} 
is Lipschitz-continuous and monotone, we can apply the Browder-Minty theorem to ensure that for each $q_\gamma \in Q^{\rm{ad}}$ a unique solution $u_\gamma \in H_0^1 ( \Omega )$ exists. 

The existence of an optimal solution to the regularized problem can be
proven by analogous arguments as in Theorem~\ref{thm:existence} providing:
\begin{theorem}
  There exists at least one solution for~\eqref{eq:Regularized-Obstacle-Problem}.
\end{theorem}

Further, the regularization allows us to formulate a set of optimality conditions for this problem, see, e.g.,~\cite{DeckelnickHinze:2011a}.
\begin{proposition}\label{prop:OptSystemRegularized}
  Let $\left( \bar{q}_{\gamma} , \bar{u}_{\gamma}\right) \in Q^{\rm{ad}} \times H^1_0(\Omega)$ be a
  local minimum of~\eqref{eq:Regularized-Obstacle-Problem}. Then there exists
  $\bar{p}_{\gamma} \in H_0^1(\Omega)$
  such that
  \begin{subequations}
    \begin{align*}
      - \nabla \cdot \left( \bar{q}_\gamma \nabla \bar{u}_{\gamma} \right)
      &= f - r ( \gamma, \bar{u}_{\gamma})& & \text{in }H^{-1}(\Omega),\\
      - \nabla \cdot \left( \bar{q}_{\gamma} \nabla \bar{p}_{\gamma} \right)
      &= \bar{u}_{\gamma} - u_d- \partial_u r( \gamma, \bar{u}_{\gamma}) \bar{p}_{\gamma}& & \text{in }H^{-1}(\Omega),\\
      \left( \alpha \bar{q}_\gamma - \nabla \bar{u}_\gamma \otimes \nabla \bar{p}_\gamma \right) (q - \bar{q}_\gamma)
      &\geq 0 &&\forall q \in Q^{\rm{ad}}
    \end{align*}
  \end{subequations}
  with $\nabla u_\gamma \otimes \nabla p_\gamma$ describing the outer product of $\nabla u_\gamma$ and $\nabla p_\gamma$.
\end{proposition}

By passing to the limit with $\gamma \to \infty$ we can utilize these conditions to formulate optimality conditions for the original problem~\eqref{eq:opt}. First we consider limits of the variables corresponding to the solutions of the regularized Problem~\eqref{eq:Regularized-Obstacle-Problem}.

\begin{theorem}\label{thm:OptimalityConditions}
  If $\gamma \to \infty$, then there is a subsequence of solutions $(\bar{q}_\gamma , \bar{u}_\gamma)$ to problem~\eqref{eq:Regularized-Obstacle-Problem}, with corresponding adjoint $p_{\gamma} \in H_0^1(\Omega)$ as defined in Proposition~\ref{prop:OptSystemRegularized}, such that
  \begin{subequations}
    \begin{align*}
      \bar{q}_{\gamma} &\to \bar{q}  &&\text{ in } L^p \left( \Omega, \mathbb{R}^{2 \times 2}_{\text{sym}} \right) \text{ for all } 2 \leq p < \infty,\\
      \bar{u}_\gamma &\to \bar{u} &&\text{ in } W^{1,p} \left( \Omega \right) \text{ for an } 2 < p < \infty,\\
      \bar{p}_{\gamma} &\to \bar{p} &&\text{ in }  W^{1,p} \left( \Omega \right) \text{ for an } 2 < p < \infty,\\
      \bar{q}_{\gamma} \nabla \bar{u}_\gamma &\to \bar{q} \nabla \bar{u} &&\text{ in } L^2 \left( \Omega, \mathbb{R}^2 \right),\\
      \bar{q}_{\gamma} \nabla \bar{p}_{\gamma} &\to \bar{q} \nabla \bar{p} &&\text{ in } L^2 \left( \Omega, \mathbb{R}^2 \right),\\
      \nabla \bar{u}_{\gamma} \otimes \nabla \bar{p}_{\gamma} &\to \nabla \bar{u} \otimes \nabla \bar{p} &&\text{ in } L^p \left( \Omega, \mathbb{R}^{2 \times 2}_{\text{sym}} \right) \text{ for an } 1 < p < \infty ,\\
      r ( \gamma, \bar{u}_{\gamma}) &\to \bar{\lambda} &&\text{ in } H^{-1} \left( \Omega \right),\\
      \partial_u r( \gamma, \bar{u}_{\gamma}) \bar{p}_{\gamma} &\to \bar{\mu} &&\text{ in } H^{-1} \left( \Omega \right)
    \end{align*}
  \end{subequations}
  for some $(\bar{q}, \bar{u}, \bar{p}, \bar{\lambda}, \bar{\mu}) \in Q^{\rm{ad}} \times H^1_0 (\Omega) \times H^1_0 (\Omega) \times L^2(\Omega) \times L^2 (\Omega)$.
\end{theorem}
Based on these limits we can formulate a set of limiting optimality conditions for the original problem.
\begin{theorem}
  Any limit point $(\bar{q}, \bar{u}, \bar{p}, \bar{\lambda},
  \bar{\mu}) \in Q^{\rm{ad}} \times H^1_0 (\Omega) \times H^1_0
  (\Omega) \times L^2(\Omega) \times L^2 (\Omega)$ obtained in
  Theorem~\ref{thm:OptimalityConditions}, fulfills the first order optimality system
  \begin{subequations}
    \begin{align*}
      - \nabla \cdot \left( \bar{q} \nabla \bar{u} \right) &= f -
      \bar{\lambda} & & \text{in } H^{-1}(\Omega),\\
      \bar{u} &\leq \psi &&\text{q.e. in } \Omega,\\
      \bar{\lambda} &\geq 0 &&\text{in } H^{-1}(\Omega),\\
       \left( \bar{\lambda}, \bar{u} - \psi \right) &= 0,\\
      - \nabla \cdot \left( \bar{q} \nabla \bar{p} \right) &= \bar{u} - u_d - \bar{\mu} & & \text{in } H^{-1}(\Omega),\\
      \left( \bar{p} , \bar{\lambda} \right) &= 0,\\
      \left( \bar{\mu}, \bar{u}-\psi \right) &= 0,\\
      \left( \bar{p}, \bar{\mu} \right) &\geq 0,\\
      \left( \alpha \bar{q} - \nabla \bar{u} \otimes \nabla \bar{p} \right) (q - \bar{q}) &\geq 0 &&\forall q \in Q^{\rm{ad}}.
    \end{align*}
  \end{subequations}
\end{theorem}

\section{Numerical Experiments}\label{sec:numerics}

In this section, we present numerical results on an obstacle problem with coefficient control and its regularization.
As a basis we consider an inverse problem for the estimation of matrix
coefficients in an elliptic pde as has been studied
in~\cite{DeckelnickHinze:2011a}. This is an optimal control problem
with coefficient control, which we further modify by introducing an
obstacle $\psi$, adjusting the objective and adding a barrier term to
handle the condition $q \in Q^{\rm{ad}}$.
\subsection{Example 1}\label{sec:obstaclenumericsEx1}
The resulting problem on the domain $\Omega = (-1,1)^2 \subset \mathbb{R}^2$ is then given by
\begin{equation}\label{eq:CoefficientControlWObstacle}\tag{$P^{\rm{MEst}}$}
  \begin{aligned}
    \min_{q \in L^2(\Omega, \mathbb{R}^{2 \times 2}_{\text{sym}}),u \in H_0^1 \left( \Omega \right) } ~ &J \left( q,u \right) + \beta B(q) && \\
    \text{s.t. } &- \nabla \cdot \left( q \nabla u \right) = f - \lambda &&  \text{ in } H^{-1} \left( \Omega \right),\\
    &u \leq \psi && \text{ q.e. in } \Omega,\\
    &\lambda \geq 0 && \text{ in } H^{-1} \left( \Omega \right),\\
    &\left( \lambda, u-\psi \right) = 0,
  \end{aligned}
\end{equation}
with given $\psi > 0$, $\psi \in \mathbb{R}$ and
\begin{equation}
  f(x_1,x_2)=\left( 1 - x_2^2 \right) \left( 6x_1^2 + 2 \right) + 2
  \left( 1-x_1^2 \right).
\end{equation}
To enforce $q \in Q^{\rm{ad}}$, we introduce a logarithmic barrier term
\begin{align*}
  - B(q) &= \int_\Omega \log \left( \rm{det} \left( q - q_{\min} I
    \right) \right) + \log \left( \rm{det} \left( q_{\max} I - q
    \right) \right) \,\rm{d}x,
\end{align*}
with $q_{\min} = 0.5$, $q_{\max} = 10$ into the objective with a, small,
barrier parameter $\beta >0$.
While $B$ is clearly a barrier for the admissible control set
$Q^{\rm{ad}}$ during the iterations one must assert that indeed the
iterates remain within $Q^{\rm{ad}}$ as $B$ can be finite outside of
$Q^{\rm{ad}}$. To do so, the trace of the matrices is monitored as in
two dimensions a matrix is positive definite if its determinant and
trace are positive.

The objective is given by
\begin{equation*}
  J \left( q,u \right) = \frac{1}{2} \| u - u_d \|^2_2 + \frac{\alpha}{2} \left\| q - q_d \right\|_2^2 
\end{equation*}
with desired state 
\begin{equation*}
  u_d(x_1, x_2)=(1-x_1^2)(1-x_2^2),
\end{equation*}
and desired control
\begin{equation*}
  q_d(x_1, x_2)= \begin{pmatrix}
  1 + x_1^2 & 0\\
  0 & 1
\end{pmatrix}
\end{equation*}
following~\cite{DeckelnickHinze:2011a} with the additional
introduction of $q_d$.
This is done so that without the introduction of an obstacle, the desired solution $(q_d, u_d)$ would be the optimal solution of~\eqref{eq:CoefficientControlWObstacle} with objective $J(q,u)$.

The problem setting has been implemented in C++ using the DOpElib optimization suite, see~\cite{dopelib}, which uses the deal.II finite element library, 
see~\cite{dealII94,deal2020}. For our finite element approximations, we utilize a uniform mesh dependent on refinement level $l \geq 0$ that is constructed of $2^l \times 2^l$ quadratic cells of size $h$. To compute discretized solutions $(q_h, u_h)$, 
we utilize piecewise bilinear finite elements.
\begin{figure}[h!]
  \includegraphics[width=11.5cm]{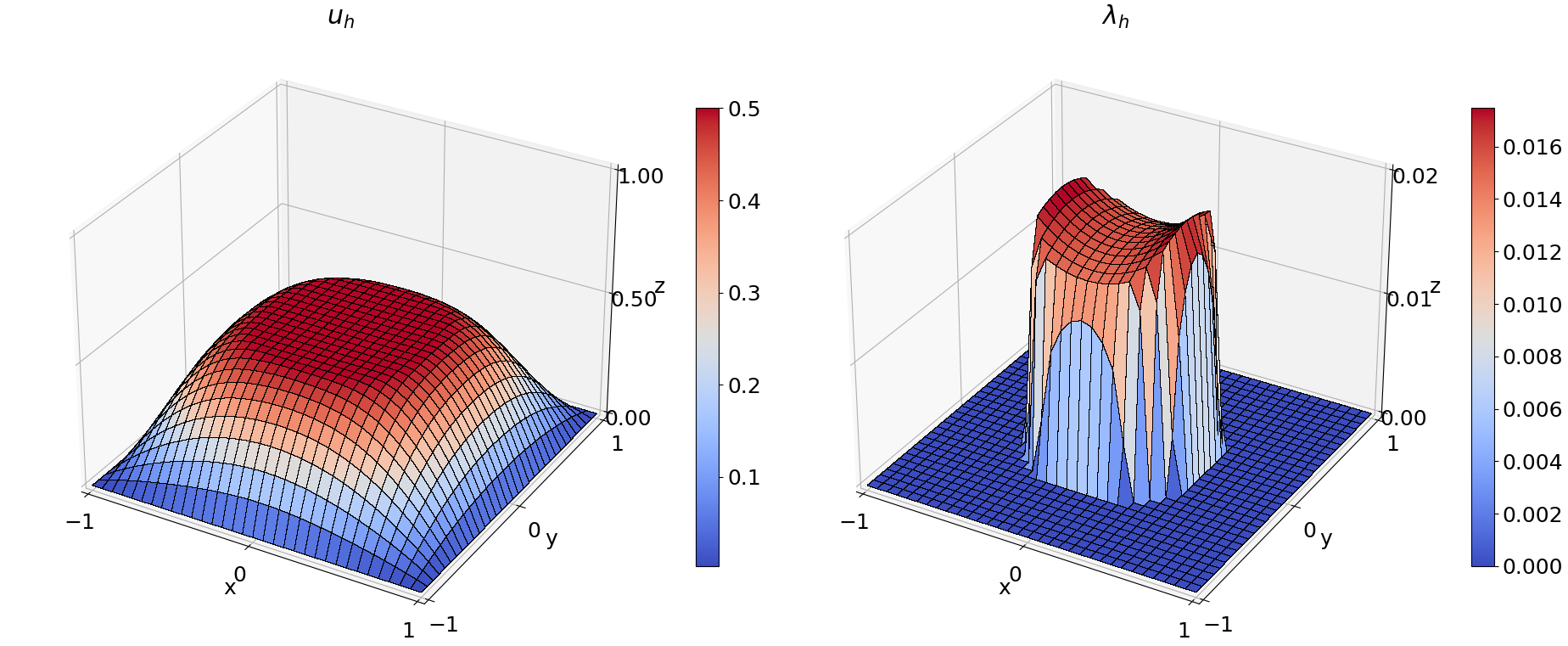}
  \centering
  \caption{State solution $(u_h,\lambda_h)$ for Problem~\eqref{eq:CoefficientControlWObstacle} at refinement level $l =5$}
  \label{fig:stateSolution}
\end{figure}
\begin{figure}[h!]
  \includegraphics[width=11.5cm]{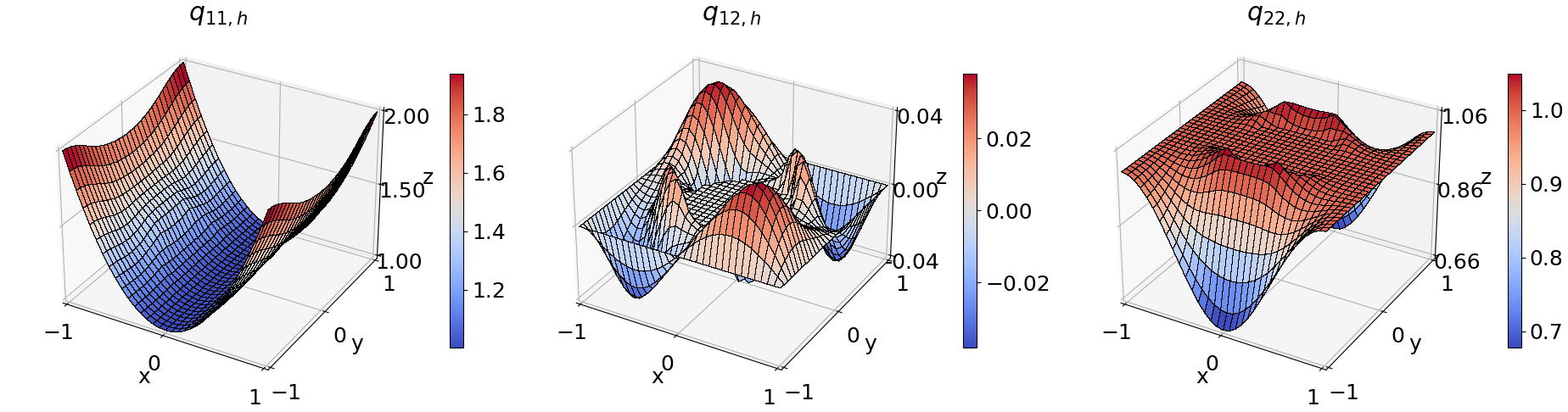}
  \centering
  \caption{Control solution $q_h$ for Problem~\eqref{eq:CoefficientControlWObstacle} at refinement level $l =5$}
  \label{fig:controlSolution}
\end{figure}

In this first problem, we compute solutions for~\eqref{eq:CoefficientControlWObstacle} on $\Omega = (-1,1)^2$ with obstacle $\psi = 0.5$ at refinement levels $l = 5$. Here we have weighted the Tikhonov term in the objective with $\alpha = 0.1$ and the barrier with $\beta = 0.0001$. We have chosen
\begin{align*}
  q^{\rm{init}} = \begin{pmatrix} 2 & -1 \\ -1 & 2 \end{pmatrix}
\end{align*}
as the initial control. To implement the obstacle we have equivalently reformulated the set of complementarity constraints into
\begin{align*}
  \lambda - \max \{0,  \lambda + c( u-\psi ) \} = 0 \iff u \leq \psi, ~ \lambda \geq 0, ~ \left( \lambda, u- \psi\right) = 0,
\end{align*}
for arbitrary $c > 0$. In the computations, we have chosen this parameter as $c=1$.\\
Figure~\ref{fig:stateSolution} shows the state solution $u_h$ and the
associated Lagrange multiplier $\lambda_h$ of this problem. We can
observe the obstacle acting as a constraint on the state, preventing
the state $u_h$ from achieving the desired solution $u_d$. Note that, since the Lagrange-multiplier $\lambda_h$ acts as a slack variable, it allows us to observe the area in which the obstacle constraint is active. The effects on the corresponding control solution $q_h$ are illustrated in Figure~\ref{fig:controlSolution}. 


\subsection{Example 2}\label{sec:obstaclenumericsEx2}
To study the effects of the regularization, we use the regularized problem formulation. It is given by
\begin{equation}\label{eq:RegCoefficientControlWObstacle}\tag{$P^{\rm{MEst}}_{\gamma}$}
  \begin{aligned}
    \min_{q_\gamma \in L^2(\Omega),u_\gamma \in H_0^1 \left( \Omega \right) } ~ &J \left( q_\gamma,u_\gamma \right) && \\
    \text{s.t. } - \nabla \cdot &\left( q_\gamma \nabla u_\gamma \right) + \gamma \max \left( u_\gamma - \psi, 0 \right)^3 = f &&  \text{ in } H^{-1} \left( \Omega \right)\\
  \end{aligned}
\end{equation}
with penalty parameter $\gamma > 0$. All other quantities as the domain $\Omega$, the parameters $\alpha,\beta$, and the obstacle $\psi$ are chosen as in Problem $\eqref{eq:CoefficientControlWObstacle}$. 
\begin{figure}[h!]
  \includegraphics[width=11.5cm]{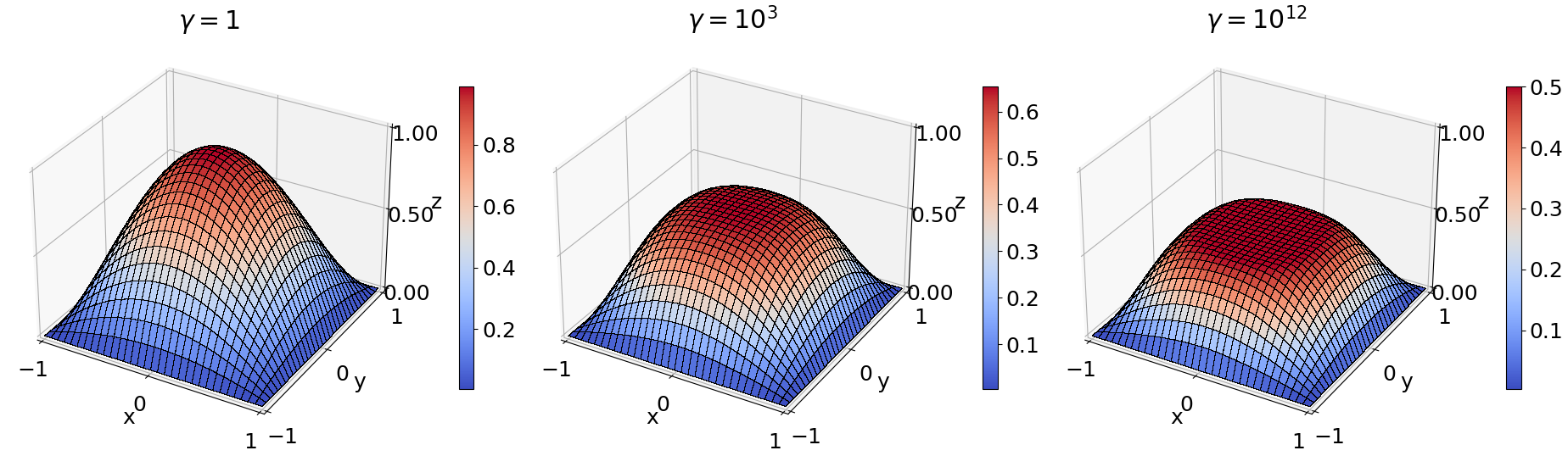}
  \centering
  \caption{Results for different choices of regularization parameter $\gamma$ on state solution $u_{h, \gamma}$ of Problem~\eqref{eq:RegCoefficientControlWObstacle} at refinement level $l = 5$}
  \label{fig:regEffect}
\end{figure}
\begin{figure}[h!]
  \includegraphics[width=11.5cm]{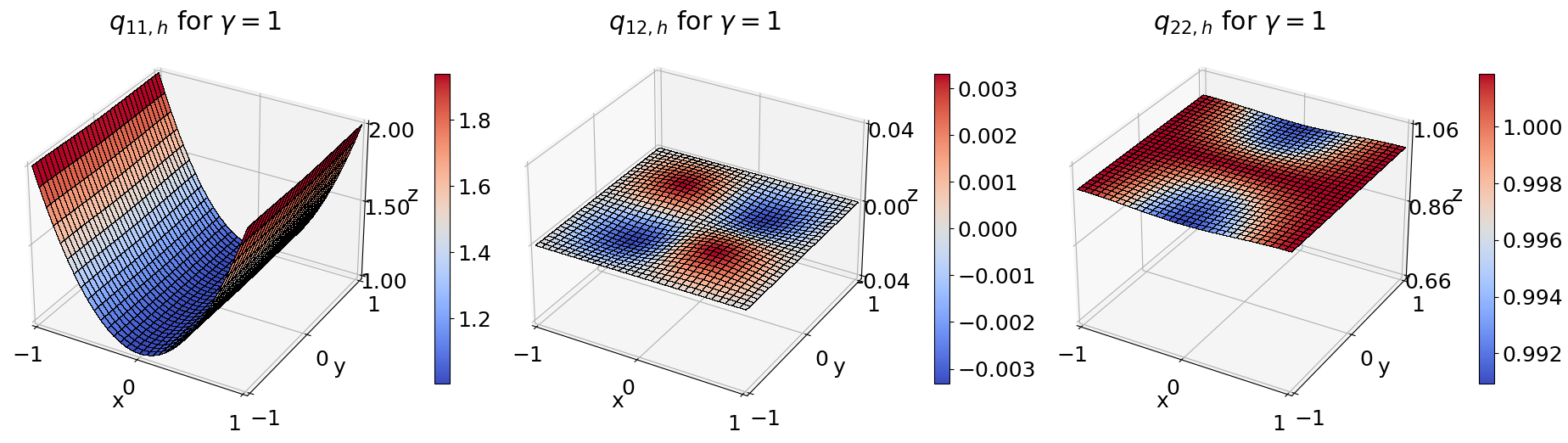}
  \centering
  \includegraphics[width=11.5cm]{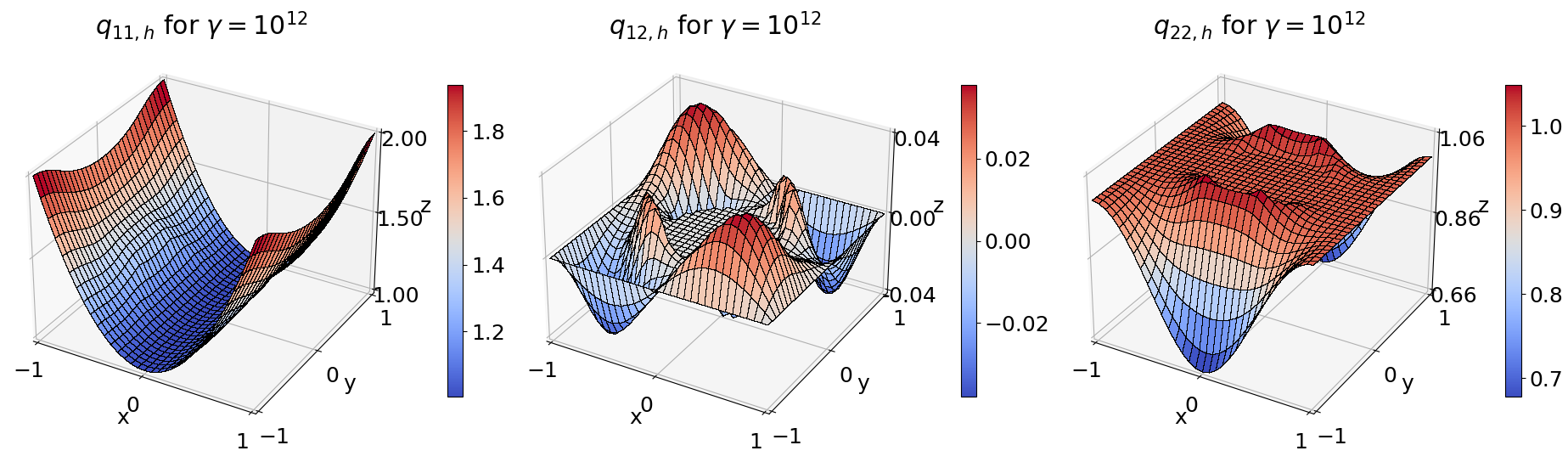}
  \centering
  \caption{Results for different choices of regularization parameter $\gamma$ on control solution $q_{h, \gamma}$ of Problem~\eqref{eq:RegCoefficientControlWObstacle} at refinement level $l = 5$.}
  \label{fig:regEffectControl}
\end{figure}

We also start our computations with the same initial control $q^{\rm{init}}$. We
now compute solutions $(q_{h, \gamma}, u_{h, \gamma})$
for~\eqref{eq:RegCoefficientControlWObstacle} at different values of
the regularization parameter $\gamma$. In Figure~\ref{fig:regEffect},
we have illustrated the impact of increasing $\gamma$ on the optimal
state solution, visibly enforcing the obstacle for higher values of
$\gamma$. We can also observe that the regularized control, see
Figure~\ref{fig:regEffectControl}, approximates the control solution
of obstacle problem~\eqref{eq:CoefficientControlWObstacle} for
increasing $\gamma$. This is supported by our numerical results when
comparing solutions of
Problem~\eqref{eq:RegCoefficientControlWObstacle} at
different regularization values with those of Problem~\eqref{eq:CoefficientControlWObstacle}, see Table~\ref{tb:Regularization-vs-Refinements} for error computations at different refinement levels.

\begin{table}[h!]
  \begin{center}
    \begin{tabular}{|c||c c | c c|} 
      \hline
      $\gamma$ & \quad   $\| u_{h_1, \gamma} - u_{h_1} \|_2$ \quad &  \quad $\| q_{h_1, \gamma} - q_{h_1}\|_2$ \quad & \quad $\| u_{h_2, \gamma} - u_{h_2} \|_2$  \quad & $\| q_{h_2,\gamma} - q_{h_2}\|_2$\\ [0.5ex] 
      \hline\hline
      $10^{0}$ & $3.79579 \cdot 10^{-1}$ & $1.8927 \cdot 10^{-1}$ & $3.79979 \cdot 10^{-1}$ & $1.87395 \cdot 10^{-1}$\\ 
      \hline
      $10^{3}$ & $1.41644 \cdot 10^{-1}$ & $6.47868 \cdot 10^{-2}$ & $1.42056 \cdot 10^{-1}$ & $6.98082 \cdot 10^{-2}$\\
      \hline
      $10^{6}$ & $1.59084 \cdot 10^{-2}$ & $7.76552 \cdot 10^{-3}$ & $1.62702 \cdot 10^{-2}$ & $8.16446 \cdot 10^{-3}$\\
      \hline
      $10^{9}$ & $1.6923 \cdot 10^{-3}$& $3.57287 \cdot 10^{-3}$ & $1.62867 \cdot 10^{-3}$ & $1.56167 \cdot 10^{-3}$\\
      \hline
      $10^{12}$ & $2.0242 \cdot 10^{-4}$ & $4.39335 \cdot 10^{-4}$ & $1.81648 \cdot 10^{-4}$ & $1.19951 \cdot 10^{-3}$\\
      \hline
    \end{tabular}
  \end{center}
  \caption{Difference between solution $(q_{h_i},u_{h_i})$ of Problem~\eqref{eq:CoefficientControlWObstacle} and solution $(q_{h_i, \gamma},u_{h_i,\gamma})$ of regularized Problem~\eqref{eq:RegCoefficientControlWObstacle} with $i=1,2$ for refinement levels $l_1 = 5$ and $l_2 = 7$.}
  \label{tb:Regularization-vs-Refinements}
\end{table}

\fracture{\section{Extension to Phase-Field Fracture}\label{sec:fracture}}

Let us introduce a free material optimization problem in the setting of fracture propagation, that is inspired by~\cite{HaslingerKocvaraLeugeringStingl:2010}. The overall goal is to control the behavior of the fracture by optimizing the stiffness tensor with a control in the coefficients and thus achieving a desired crack pattern.

\subsection{Problem Statement}
The state of the material is given by a pair $\boldu = (u, \varphi)$, where $u$ denotes a two dimensional displacement field and $\varphi$ a phase-field, i.e., a smooth indicator function for the fracture, cf.,~\cite{BourFraMar08, BourFraMar00}, with $\varphi = 0$ in the broken area, and $\varphi = 1$ in the intact area. The symmetric gradient $e(u)$ is defined as
\begin{align*}
  e_{ij}(u(x)) \coloneqq \frac{1}{2} \Big( \frac{du_i(x)}{dx_j} + \frac{du_j(x)}{dx_i} \Big), \quad i, j = 1, 2,
\end{align*}
and the strain by ${\sigma_q}_{ij} \coloneqq q_{ijkl} e_{kl}(u), \, i,j,k,l = 1,2$,
where $q_{ijkl}$ is the elastic/plane-stress stiffness tensor. In accordance with~\cite{HaslingerKocvaraLeugeringStingl:2010}, it is written as the symmetric material matrix
\begin{align}\label{Qdef}
  {q} = & \,  \begin{pmatrix}
    q_{1111}    & q_{1122} & \sqrt{2}q_{1112}\\
    & q_{2222} & \sqrt{2}q_{2212}\\
    \text{sym}  &          & 2 q_{1212}
  \end{pmatrix}.
\end{align}

As spatial domain $\Omega$, we choose the unit square $(0,1)^2 \subset
\mathbb{R}^2$ with a horizontal notch in the middle of the domain $[0,1]\times\{0.5\}$. Moreover, the Lipschitz boundary is partitioned into $\partial \Omega \coloneqq \Gamma_D \stackrel.\cup \Gamma_N \stackrel.\cup \Gamma_{\text{free}}$, where $\Gamma_D \coloneqq [0,1] \times \{0\}$ and $\Gamma_N \coloneqq [0,1] \times \{1\}$. 
On $\Gamma_D$ we enforce homogeneous Dirichlet boundary conditions for
the displacement $u$. We have homogeneous Neumann boundary conditions
for the phase-field and the initial condition $\varphi \equiv 1$ in
$\Omega$ at $t=0$.
Further, let
$\mathbf{f}_{|\Gamma_N} \in L^2(\Gamma_N)$ be a stationary external orthogonal
force. We consider a time-discrete model formulation on the time interval $[0,1]$ with $M+1$ equidistant time points, i.e., $0 = t_0 < t_1 < ... < t_M = 1$. Consequently the state is given by $\boldu = (\boldu^i)^M_{i=1} = (u^i, \varphi^i)^M_{i=1}$. However, we will assume both the material matrix $q$ and the external load $\mathbf{f}$ to be constant in time. 
The fracture problem then reads: For given material matrix $q \in L^2(\Omega, \R^{3 \times 3}_{sym})$, initial values $\boldu^0 \in V$ and right-hand side $\mathbf{f}$, find a state $\boldu^i \in V := H^1_D(\Omega; \R^2) \times H^1(\Omega)$ that solves for all $i = 1,...,M$,
\begin{equation}
  \label{PDEFrac}
  \tag{PDE$^{\rm{Frac}}$}
  \langle A(\boldu^i, q),\boldv\rangle + \langle
  R(\varphi^i;\gamma),v^\phi\rangle = \langle
  \mathbf{f},v^u\rangle_{\Gamma_N}, \quad \forall \boldv =
  (v^u,v^\phi) \in V.
\end{equation}
The operators $A(\boldu^i, q)$ and $R(\varphi^i;\gamma)$ are given by
\begin{align*}
  \langle A(\boldu^i, q), \boldv \rangle := & \, \Big( \big( ( 1 - \kappa)(\varphi^i)^2 + \kappa \big) \sigma_q(u^i), e(v^u) \Big)  \\
  & \, + \varepsilon G_c \big( \nabla \varphi^i, \nabla v^\varphi \Big) - \frac{G_c}{\varepsilon} \Big( 1 - \varphi^i, v^\varphi \Big) + \eta \Big( \varphi^i - \varphi^{i-1}, v^\varphi \Big) \\
  & \, + \Big( (1 - \kappa) \varphi^i \sigma_q(u^i) \colon e(u^i), v^\varphi \Big), \\
  \langle R(\varphi^i; \gamma), v^\varphi \rangle := & \, \Big(\gamma \max(0, \varphi^i - \varphi^{i-1}), v^\varphi \Big),
\end{align*}
for any $(v^u, v^\varphi) \in V$, see,~\cite{NeitzelWickWollner:2017,NeiWiWo19}. Here, $\kappa$ denotes a
(bulk) regularization parameter that helps extending the displacements
to the entire domain $\Omega$, $\varepsilon$ is a phase-field
regularization parameter, $\gamma$ is a penalty parameter for the
crack irreversibility condition $\varphi^i \leq \varphi^{i-1}$,
$\eta$ denotes a viscosity parameter, and $G_c$ is the critical energy
release rate. For further explanation on phase-field fracture, and the
physical interpretation of the involved parameters, we refer
to~\cite{NeitzelWickWollner:2017,NeiWiWo19,KhiSteiWi22_JCP}. 

We investigate an optimal control problem
with tracking type cost functional $J$. The objective is to reach a
given desired crack pattern $\varphi_d \in V$ as well as a desired material
matrix $q_d \in L^2(\Omega, \R^{3 \times 3}_{\rm{sym}})$. With
constraints given
by~\eqref{PDEFrac}, the optimal control problem reads: 

\begin{equation}\label{eq:Fracture-Problem}\tag{P$^{\rm{Frac}}$}
  \begin{aligned}
    \min_{q, \boldu} J(q, \boldu) := \sum_{i=1}^M&\Big( \frac{1}{2} \norm{ \varphi^i - \varphi_d}^2 + \frac{\alpha}{2} \norm{q - q_d}_2^2 + \beta B(q) \Big), \\
    \text{s.t. } & \quad \boldu^i \text{ and $q$ satisfy } \text{(PDE$^{Frac}$) for all } i =1,...,M,
  \end{aligned}
\end{equation}
where $\alpha > 0$ is a Tikhonov cost parameter, and $\beta > 0$ is a barrier parameter. The barrier function is defined by 
\begin{align*}
  -B(q) \coloneqq &\int_\Omega \log{(q_{1:1,1:1} - {q_L}_{1:1,1:1})} + \log{({q_U}_{1:1,1:1} - q_{1:1,1:1})}\\
  &+\log{\det(q_{1:2,1:2} - {q_L}_{1:2,1:2})} + \log{\det({q_U}_{1:2,1:2} - q_{1:2,1:2})}\\
  &+\log{\det(q - q_L)} + \log{\det(q_U - q)}\,\rm{d}x,
\end{align*}
where 
$q_{1:1,1:1}$, $q_{1:2,1:2}$ are the leading principal submatrices of
the control matrix $q$ defined in~\eqref{Qdef}. Further, $q_L =
q_{\min} I \in \R^{3\times 3}$, ${q_L}_{1:2,1:2} = q_{\min} I\in
\R^{2\times 2}$ and ${q_L}_{1:1,1:1} = q_{\min} \in \R$, respectively
for $q_U$, etc. Note that the integrand in the barrier is finite if
and only if $q-q_L$ and $q_U-q$ are positive definite and thus the
control fulfills the constraints specified in $\Qad$, similar to
Problem~\eqref{eq:CoefficientControlWObstacle} but without the need to
check for values $q$ outside of $\Qad$.

We conduct two numerical test examples, which are both motivated by
the single edge notched tension test~\cite{MieWelHof10b,
  MieWelHof10a}. The propagating fracture is caused by a constant
orthogonal force $\mathbf{f}_{|\Gamma_N} = \begin{pmatrix}0, & 2100
\end{pmatrix}^T$. In both examples, we chose the time interval $[0,1]$, with $501$ and $101$ equidistant time points in Example 1, and Example 2, respectively. The spatial mesh has $64 \times 64$ square elements.

\subsection{Example 1: Material Susceptible to Fracture Propagation}

In this first example the initial control is defined as
\begin{align*}
  q^{\rm{init}} = \begin{pmatrix}
    2 \mu_1 + \lambda_1 & \lambda_1 & 0\\
    \lambda_1 & 2 \mu_1 + \lambda_1 & 0\\
    0 & 0 & 2 \mu_1
  \end{pmatrix},
\end{align*} 
which represents the standard elasticity tensor for the Lam\'e parameters $\lambda_1 = \frac{2 \nu \mu_1}{1 - 2 \nu} $ and $\mu_1 = \frac{E}{2(1+\nu)}$, cf.,~\cite{KhiSteiWi22_JCP}.
The desired phase-field continues the initial notch to the left, i.e.,
\begin{align*}
  \varphi_d(x,y) \coloneqq
  \begin{cases}
    0, & x\in [0.25,0.5] \text{ and }y\in [0.5-0.0221, 0.5+0.0221]\\
    1, &\text{ else}.
  \end{cases}.
\end{align*}
The desired control $q_d$ is defined as 
\begin{align*}
	q_d&=q^{\rm{init}}\quad &\text{in } [0.45,1]\times [0,1],\\
  q_d &= \begin{pmatrix}
    2 \mu_2 + \lambda_2 & \lambda_2 & 0\\
    \lambda_2 & 2 \mu_2 + \lambda_2 & 0\\
    0 & 0 & 2 \mu_2
  \end{pmatrix},\quad &\text{in } [0,0.45]\times [0,1],
\end{align*}
which corresponds to the Lam\'e parameters $\lambda_2 = \lambda_1$ and
$\mu_2 = 0.01 \mu_1$ in the latter subdomain. The choice of $q_d$
describes a material that is more susceptible to fracture in the left
part of the domain. Within the optimization process we seek a control
$q$ that is closer to $q_d$ in order to get a different crack pattern,
compared to the one that we get from $q^{\rm{init}}$. In
Table~\ref{tb:FractureProblemParametersPropagation}, we present
further numerical parameters that lead to $\rm{tr}_{\Omega}(q^{\rm{init}}) := \int_{\Omega} \rm{trace}(q^{\rm{init}})\,\rm{d}x = 6\mu_1 + 2\lambda_1  = 3.056 \cdot 10^6$. In Figure~\ref{fig:EX_1_PF}, we compare the phase-fields at the final timepoint
$t_{500}$ for the initial control $q^{\rm{init}}$ and the control $q^{\rm{fin}}$, where
$\rm{tr}_{\Omega} (q^{\rm{fin}}) \approx 2.776 \cdot 10^6 $, see Figure \ref{fig:EX_1_Control} for the corresponding diagonal entries of $q^{\rm{fin}}$.
%

\subsection{Example 2: Effects of the Desired Control}
In this example, we focus on the effects of adjusting the desired
control $q^d$. On $[0.35,1] \times [0,1]$ we set $q_d = q^{\rm{init}}$. On $[0,0.35] \times [0,1]$ we adjust the Lam\'e parameters similar to Example 1, but using $\mu_2 = 100 \mu_1$. Here we chose a time interval in $[0,1]$ with $101$ equidistant time points. Mesh size, constant orthogonal force $\mathbf{f}_{|\Gamma_N} = \begin{pmatrix}0, & 2100
\end{pmatrix}^T$, and all other parameters remain the same, see  Table~\ref{tb:FractureProblemParametersPropagation}. We want to observe the effects of increasing the Lam\'e parameter $\mu_2$ on part of the domain to achieve a different crack pattern, as opposed to Example 1 where we observed the effects of decreasing this parameter. In Figure~\ref{fig:EX_2_PF}, we compare the phase fields at final timepoint $t_{100}$ for the initial control $q^{\rm{init}}$ and the control $q^{\rm{fin}}$, where $\rm{tr}_{\Omega}(q^{\rm{fin}}) \approx 1.5332 \cdot 10^8$, for the corresponding diagonal entries of $q^{\rm{fin}}$ we refer to Figure \ref{fig:EX_2_Control}.

\begin{table}[h!]
  \begin{center}
    \begin{tabular}{c c c c} 
      \hline
      Parameter & Definition & Value  \\ [0.5ex] 
      \hline\hline
      $\varepsilon$ & Regularization (crack) & 0.0884 \\ 
      \hline
      $\kappa$ & Regularization (bulk) & 1.0e-10 \\ 
      \hline
      $\eta$ & Regularization (viscosity) & 1.0e3 \\
      \hline
      $\gamma$ & Penalty & 1.0e5 \\ 
      \hline
      $\alpha$ & Tikhonov & 4.75e-4 \\
      \hline
      $G_c$ & Fracture toughness & 1.0 \\
      \hline
      $\nu$ & Poisson's & 0.2 \\
      \hline
      $E$ & Young's modulus & 1.0e6 \\
      \hline
    \end{tabular}
  \end{center}
  \caption{Parameters for Problem~\eqref{eq:Fracture-Problem}.}
  \label{tb:FractureProblemParametersPropagation}
\end{table}
\newcommand\image[2][]{\bgroup\fboxsep0pt\fboxrule0pt
  \protect\fcolorbox{black}{black!10}{\protect\includegraphics[#1]{#2}}\egroup}

\newcommand\scale[5][0.1]{\parbox{#2\linewidth}{\vspace{0.3pc}%
    \makebox[\linewidth]{$#3$\hfill$#4$\hfill$#5$}\\[-#1pc]%
    \image[trim=225 1180 275 945,clip=true,width=1.0\linewidth]
    {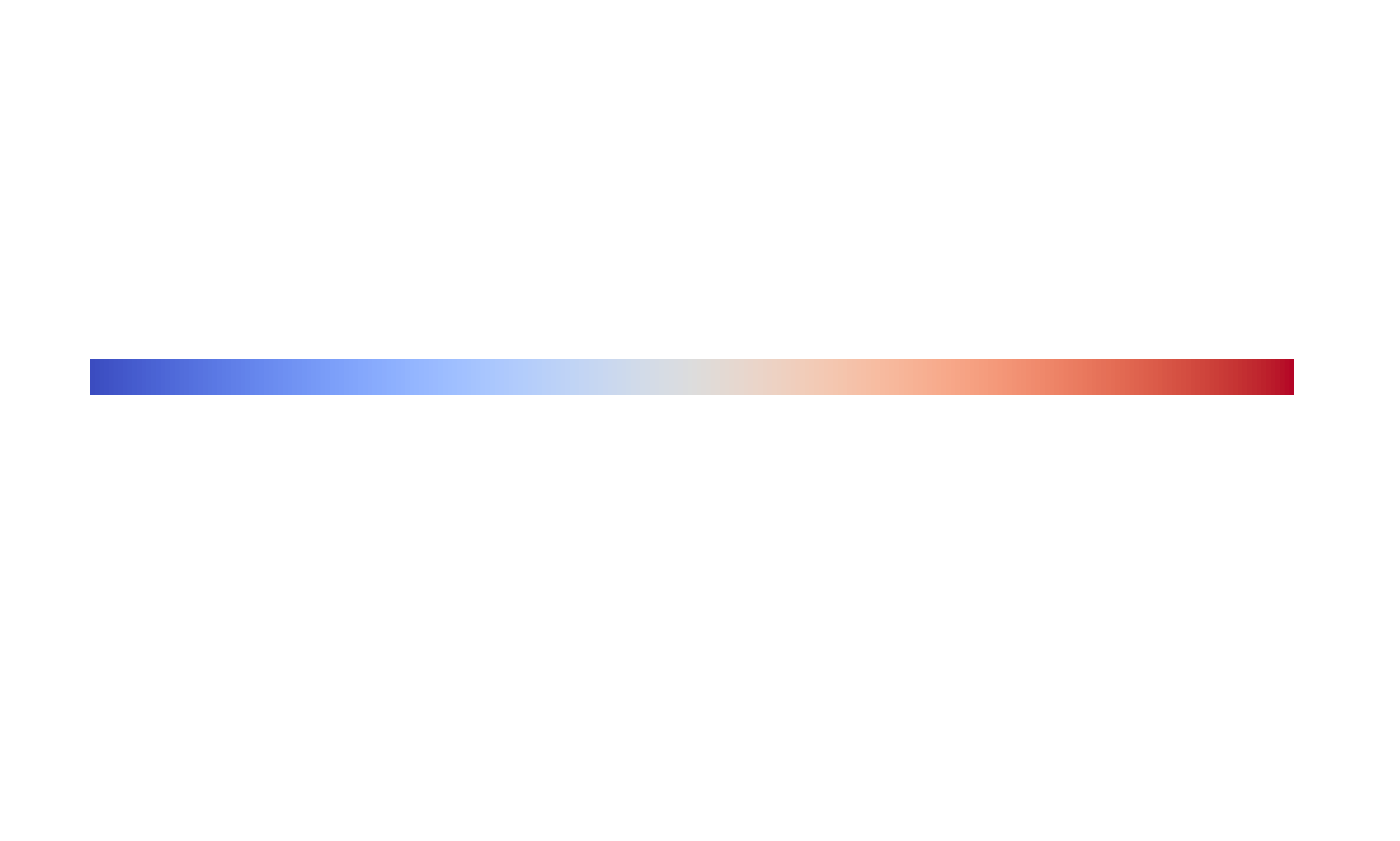}}%
}
\newcommand\img[1]{\image
  [trim=550 70 550 70,clip=true,width=.5\linewidth]
  {images/WeakMaterial_state_ts500_it#1.png}}
\begin{figure}[h!]
  \footnotesize
  \centering
  \img{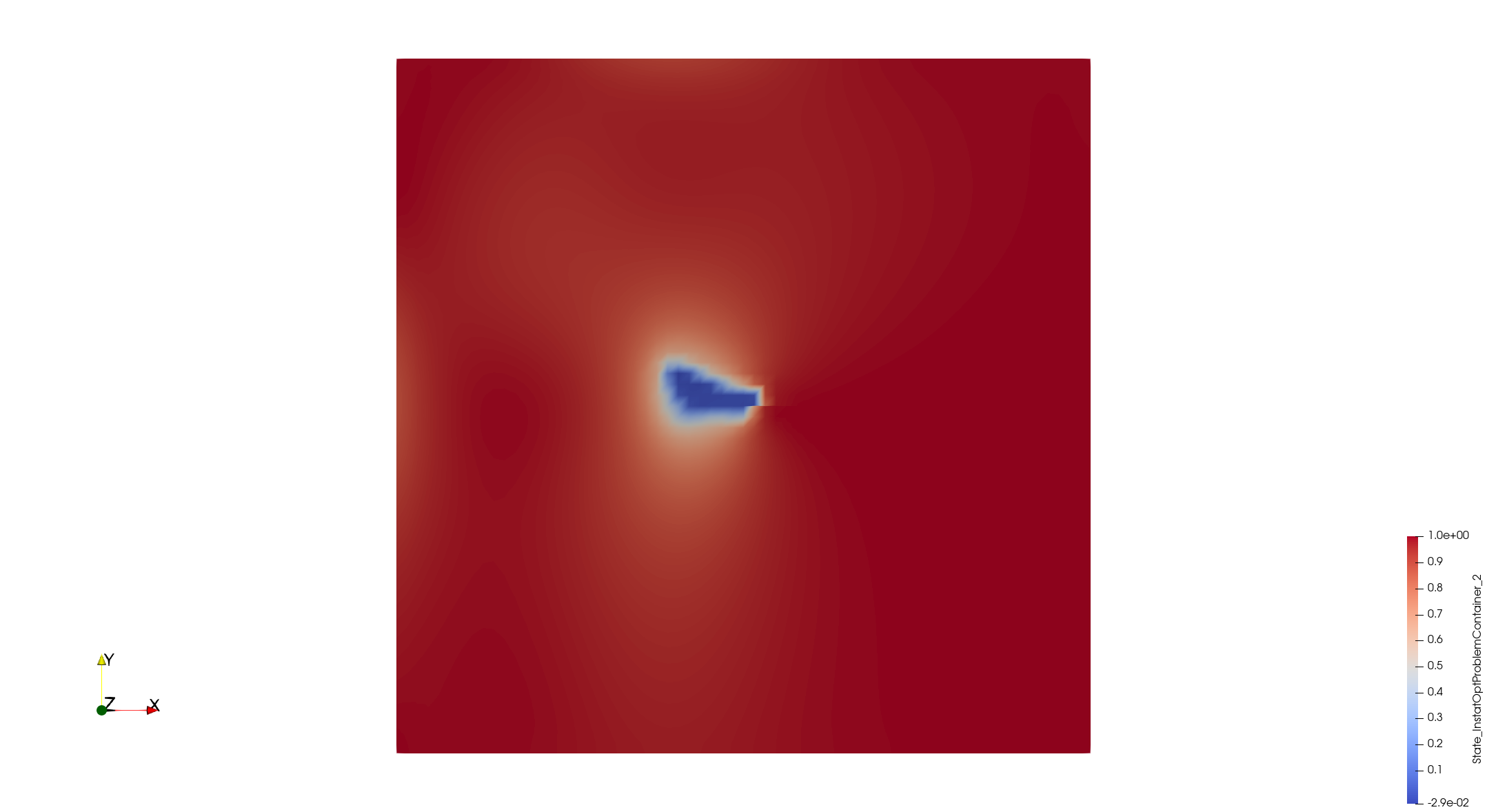}\hfill
  \img{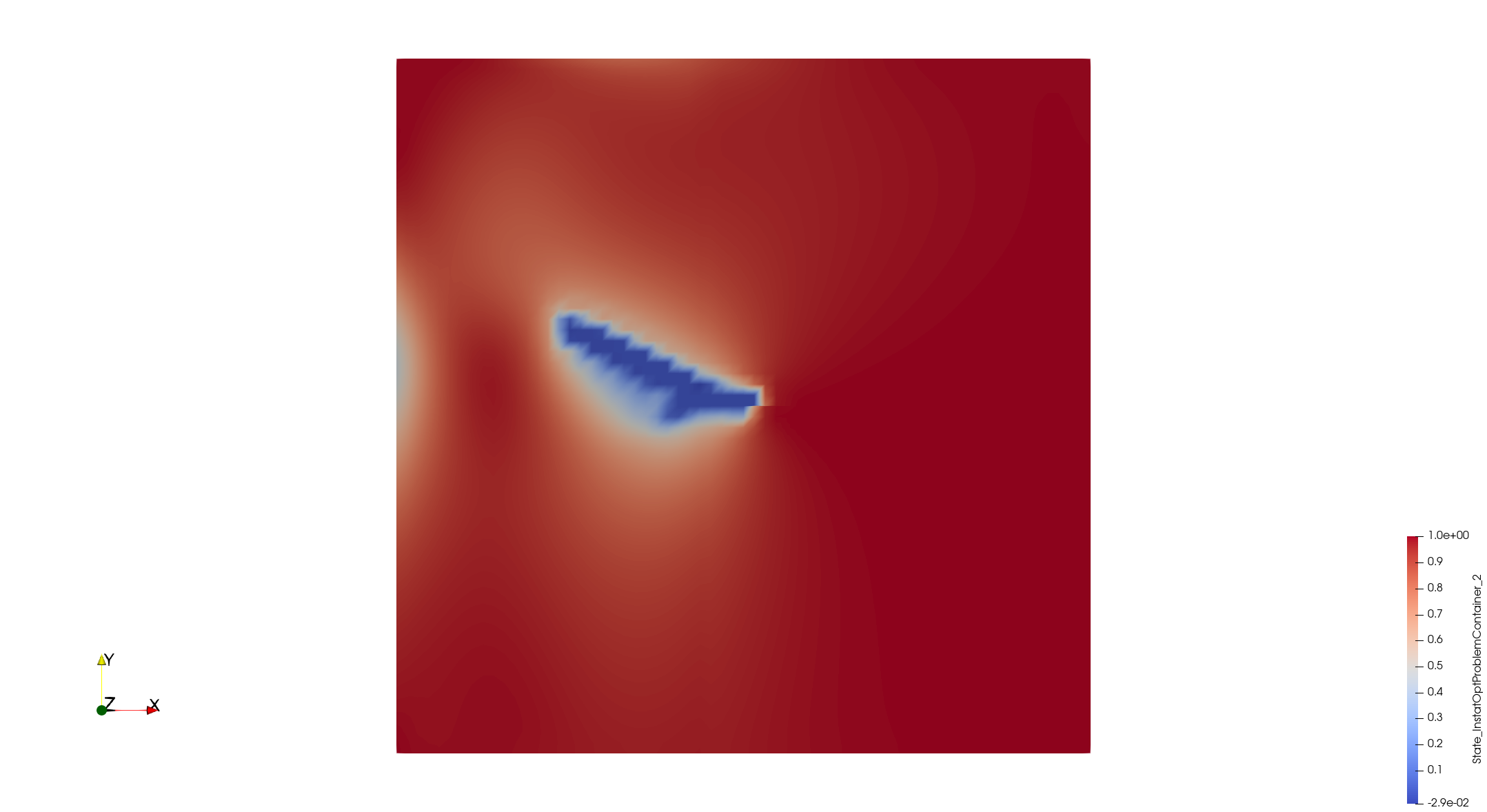}\\
  \scale[0.33]{0.32}{0}{0.5}{1}%
  \caption{Crack Pattern of Example 1 after 500 timesteps for initial control (left) and final control (right) on a $64 \times 64$ mesh}
  \label{fig:EX_1_PF}
\end{figure}

\renewcommand\img[1]{\image
  [trim=550 70 550 70,clip=true,width=.3\linewidth]
  {images/WeakMaterial_it14_control#1.png}}
\begin{figure}[h!]
  \footnotesize
  \centering
  \img{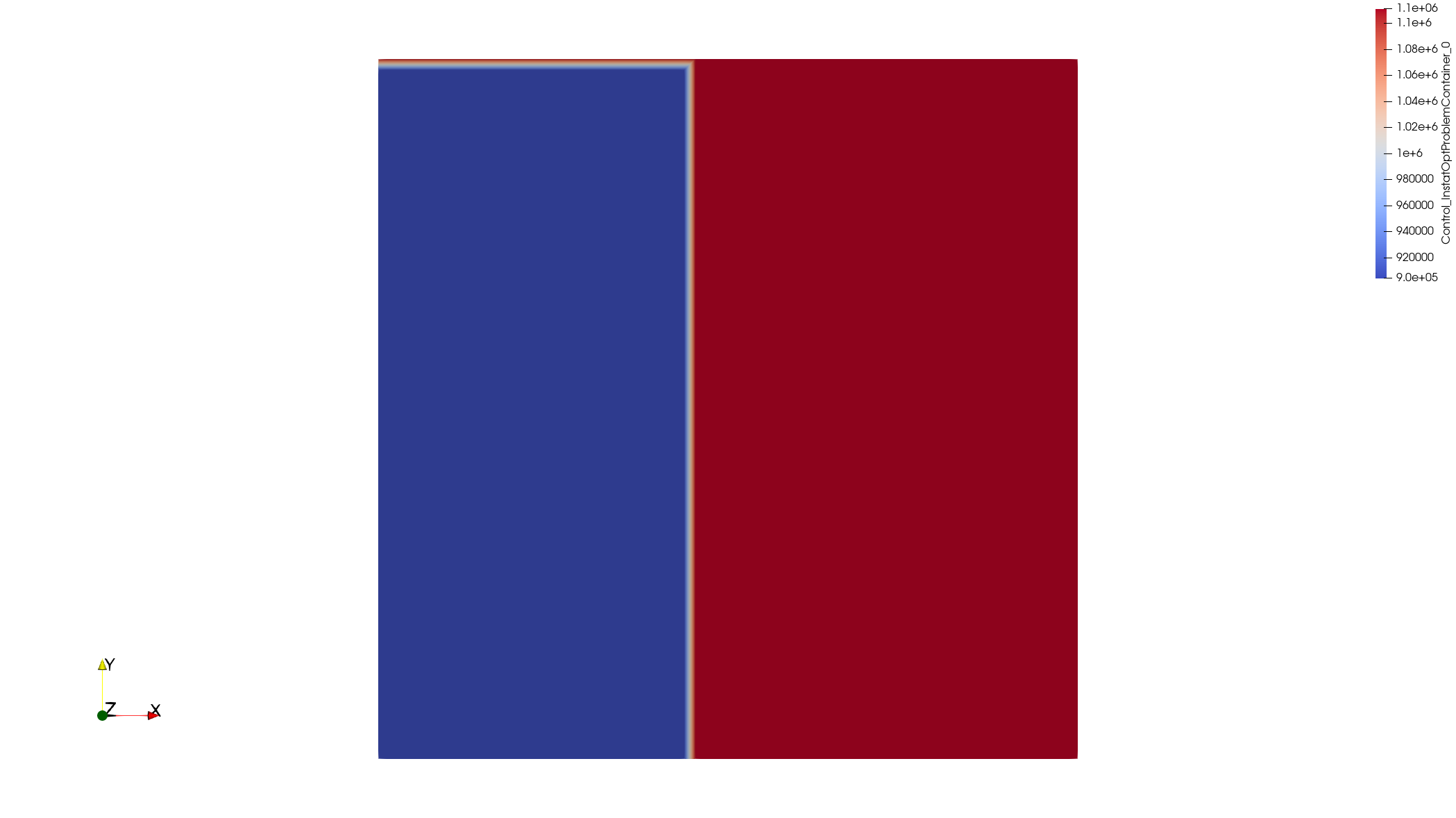}\hfill
  \img{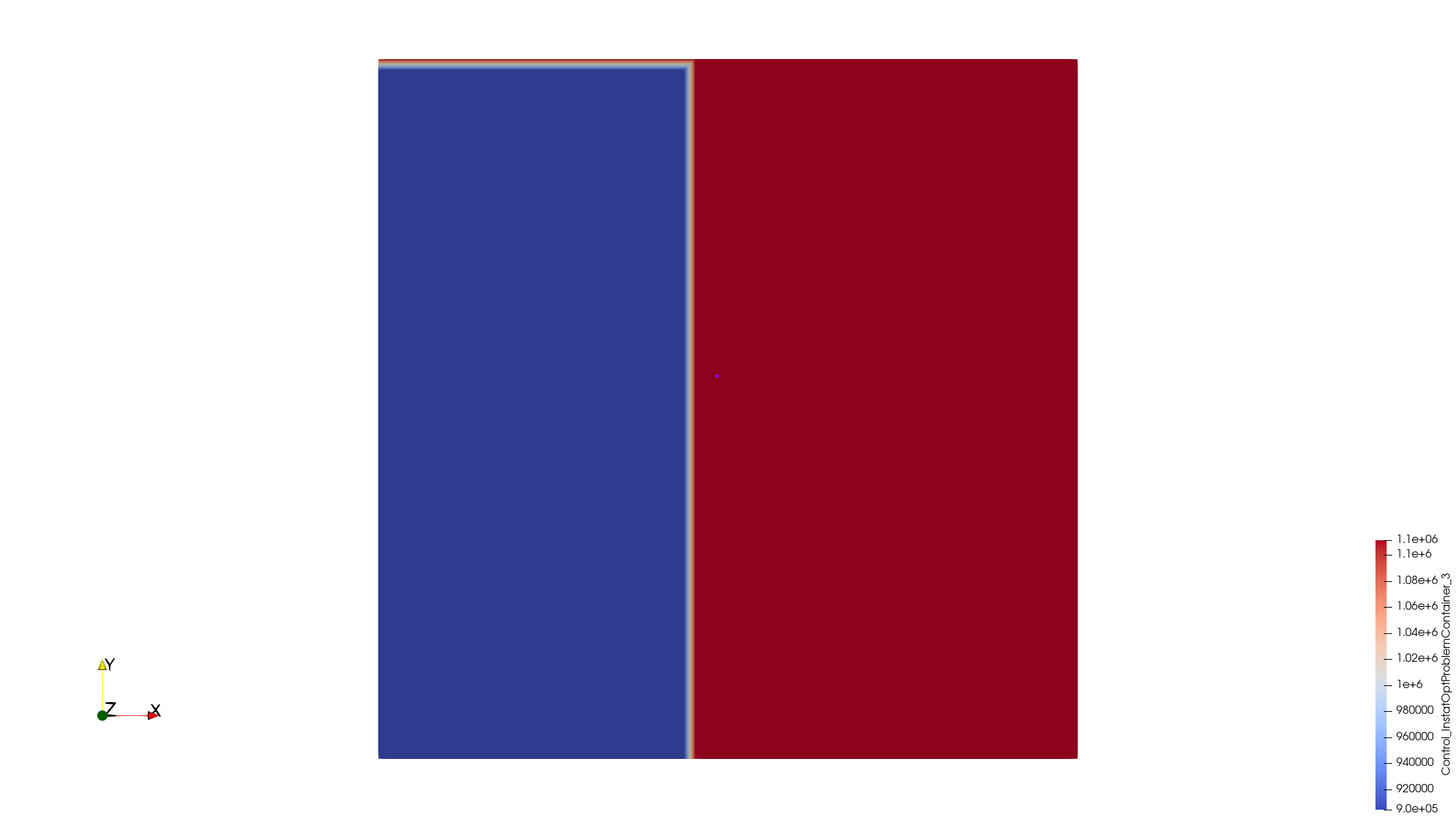}\hfill
  \img{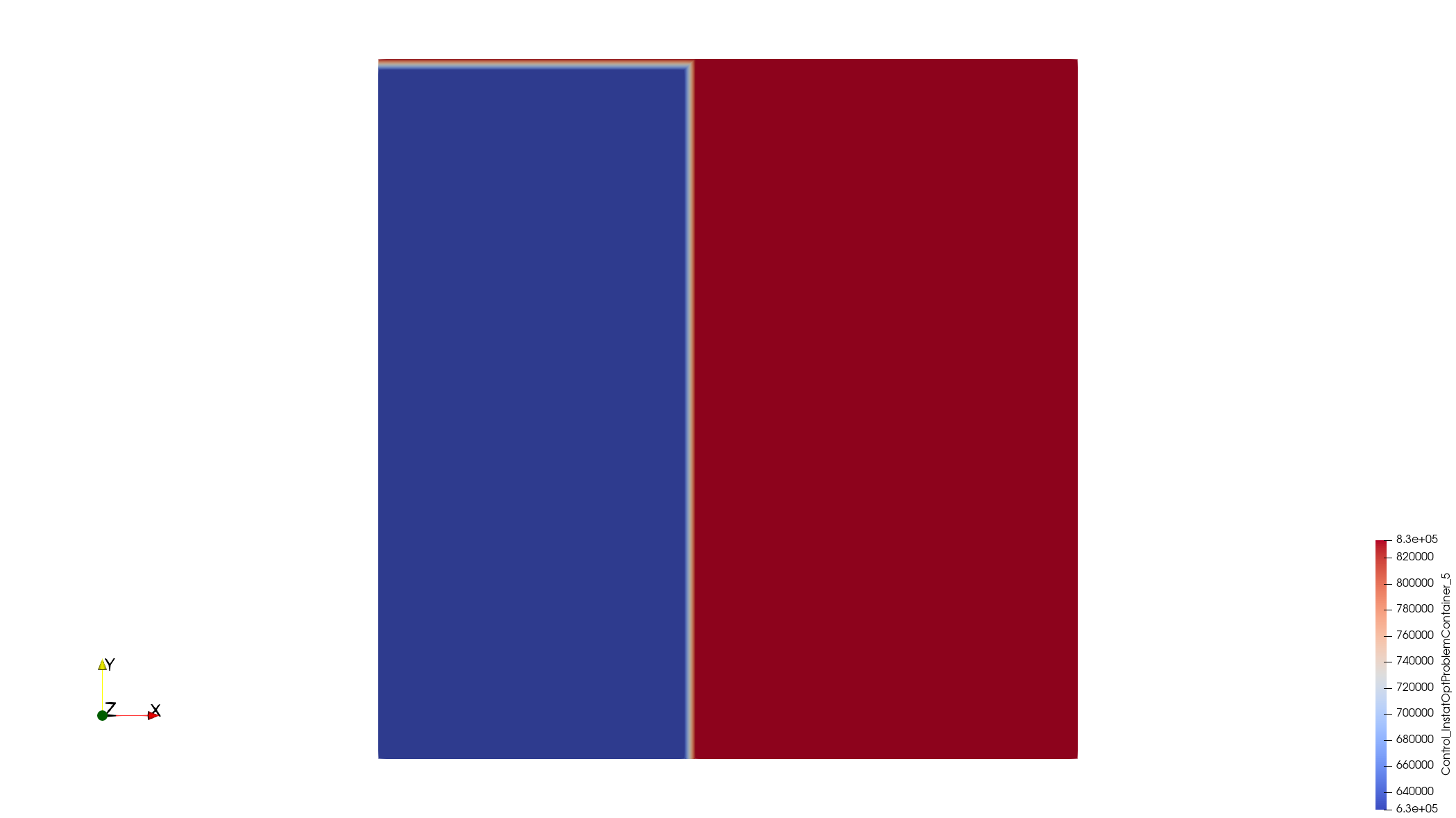}\\
  \scale[0.33]{0.3}{9 \cdot 10^5}{}{1.1 \cdot 10^6}\hfill
  \scale[0.33]{0.3}{9 \cdot 10^5}{}{1.1 \cdot 10^6}\hfill
  \scale[0.33]{0.3}{6.3 \cdot 10^5}{}{8.3 \cdot 10^5}%
  \caption{Diagonal entries of final control $q^{\rm{fin}}$ of Example 1 on a $64 \times 64$ mesh}
  \label{fig:EX_1_Control}
\end{figure}

\renewcommand\img[1]{\image
  [trim=550 70 550 70,clip=true,width=.5\linewidth]
  {images/StrongMaterial_state_ts100_it#1.png}}
\begin{figure}[h!]
  \footnotesize
  \centering
  \img{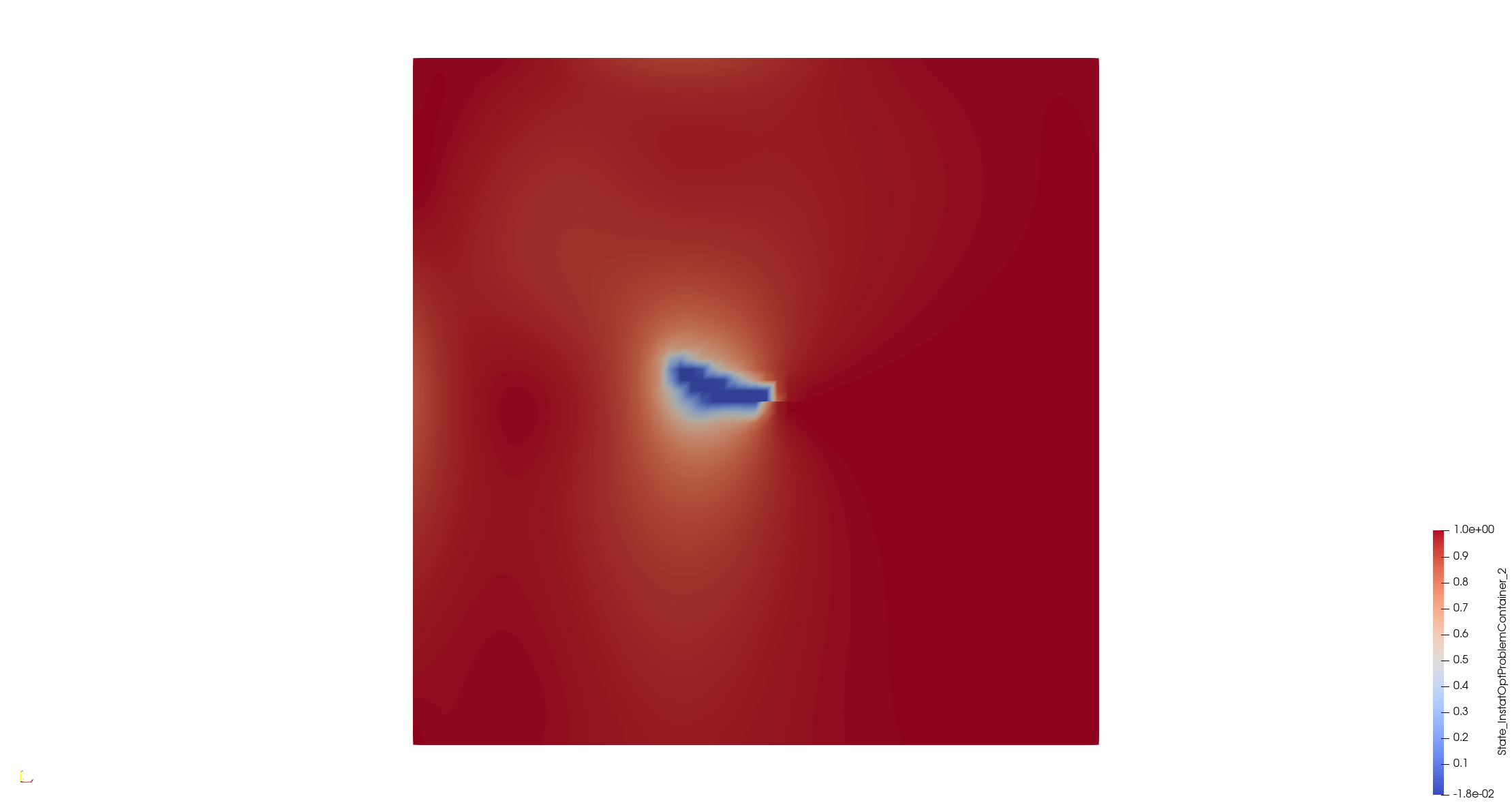}\hfill
  \img{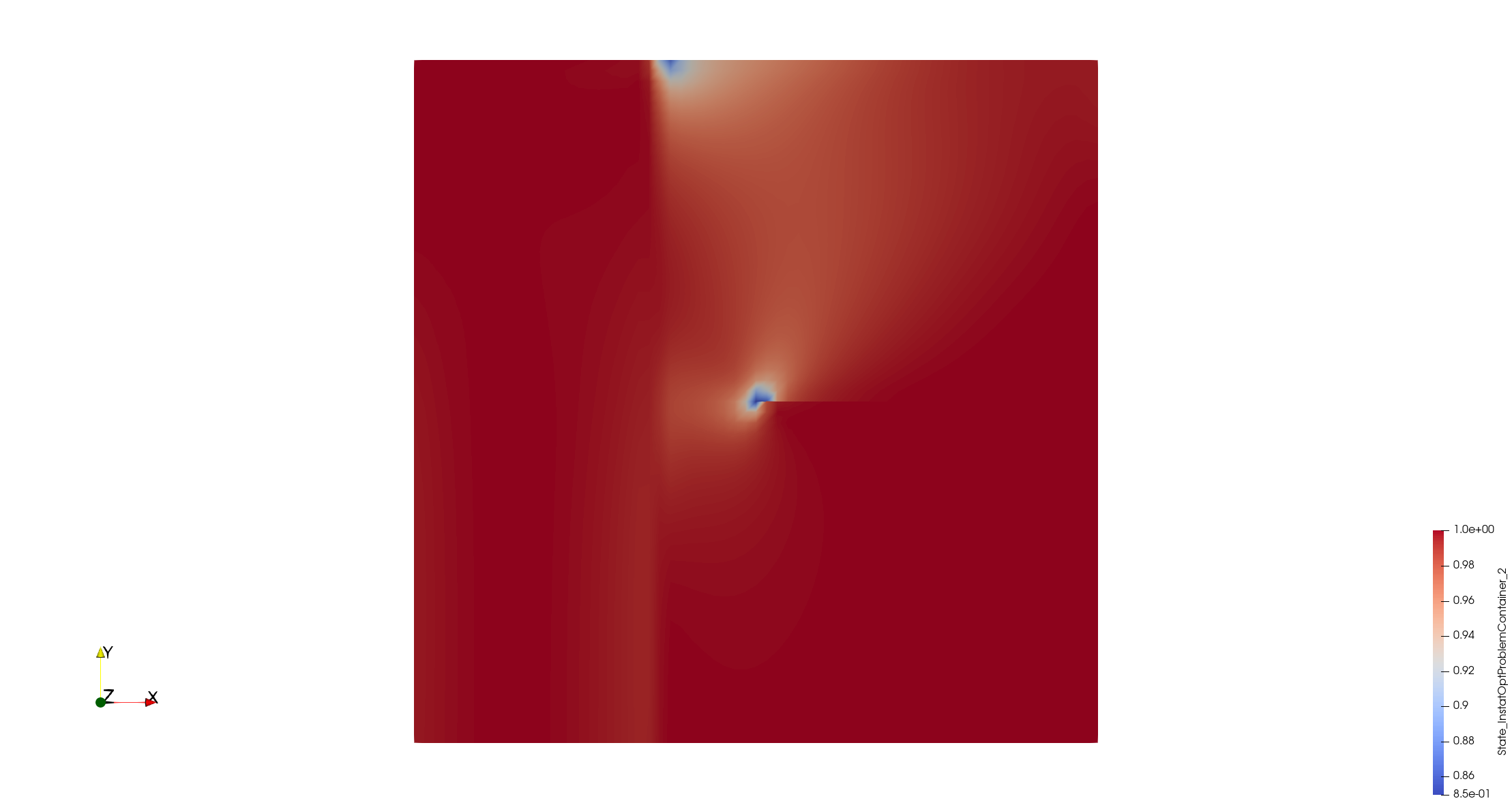}\\
  \scale[0.33]{0.32}{0}{0.5}{1}\hspace{2cm}
  \scale[0.33]{0.32}{0.85}{0.925}{1}
  \caption{Crack Pattern of Example 2 after 100 timesteps for initial control (left) and final control (right) on a $64 \times 64$ mesh}
  \label{fig:EX_2_PF}
\end{figure}
\renewcommand\img[1]{\image
  [trim=550 70 550 70,clip=true,width=.3\linewidth]
  {images/StrongMaterial_it30_control#1.png}}
\begin{figure}[h!]
  \footnotesize
  \centering
  \img{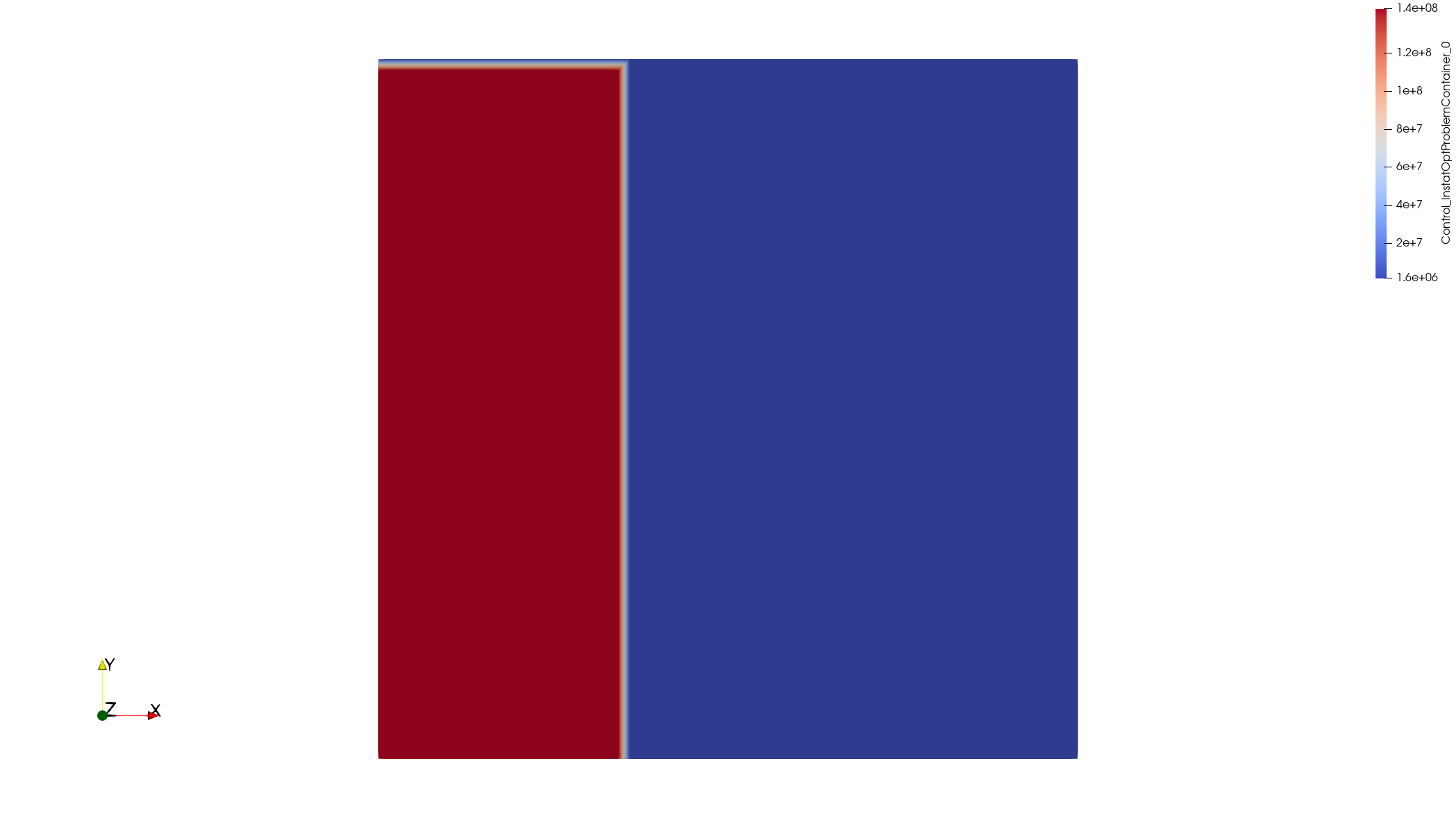}\hfill
  \img{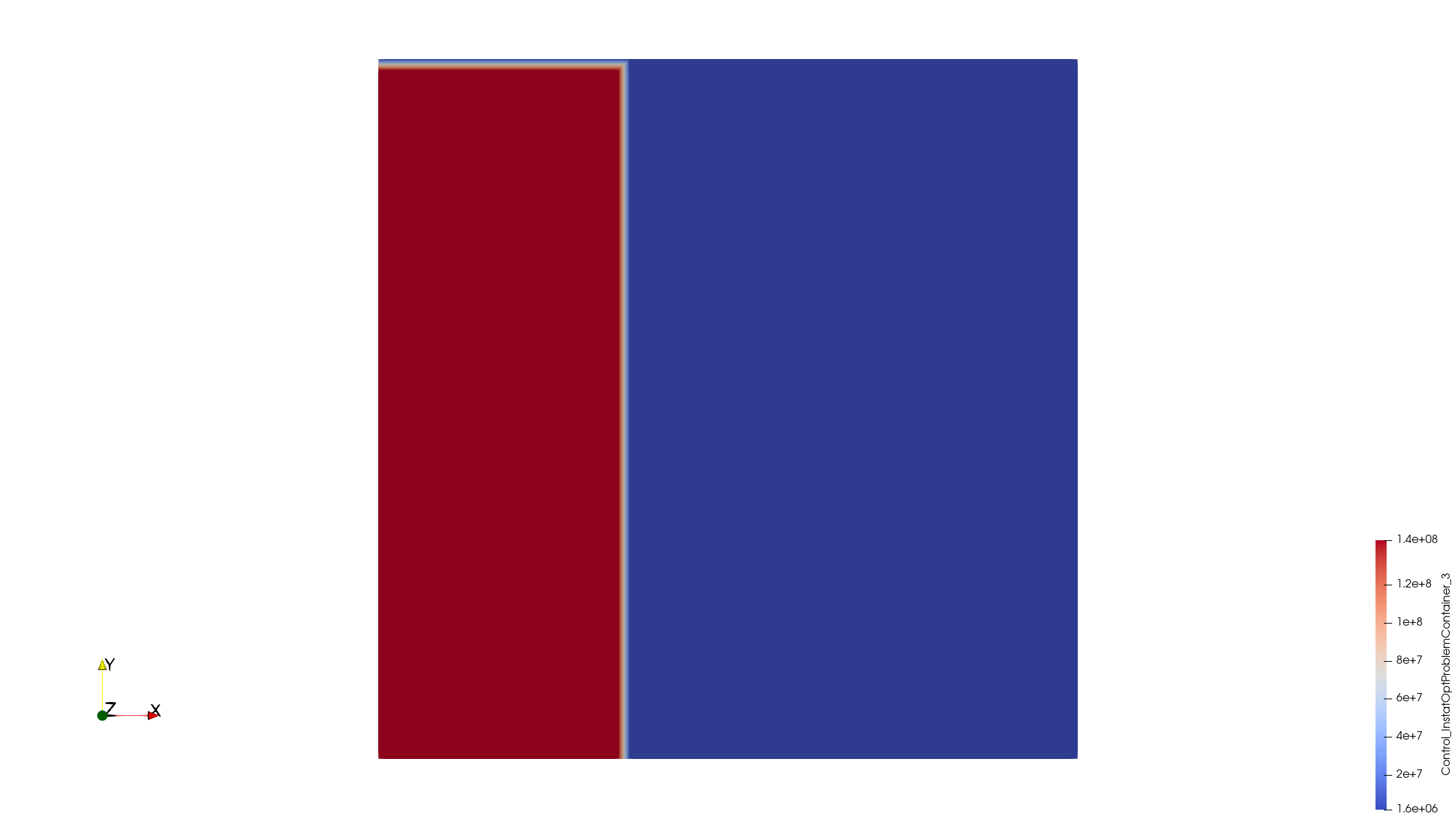}\hfill
  \img{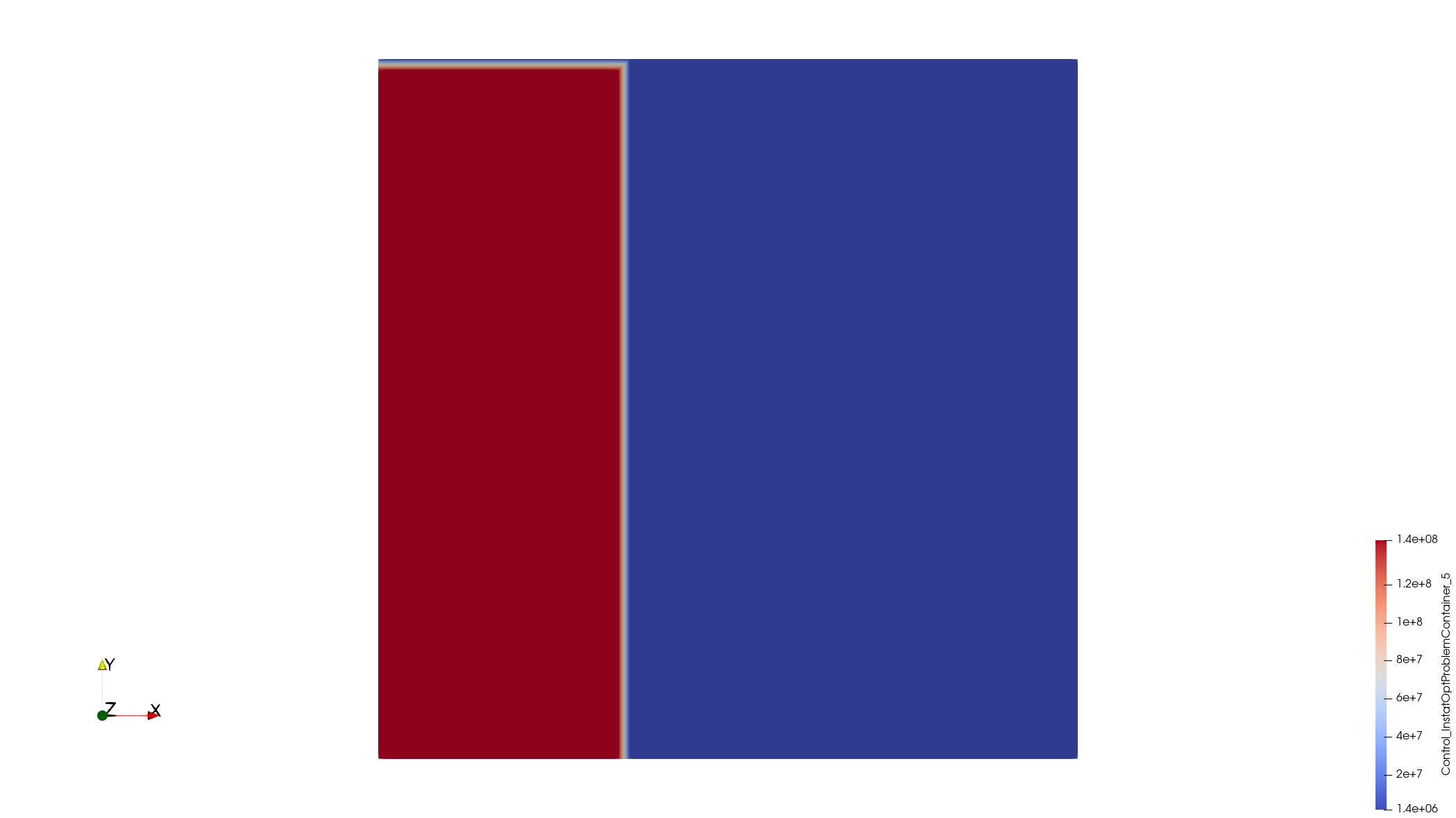}\\
  \scale[0.33]{0.3}{1.6 \cdot 10^6}{}{1.4 \cdot 10^8}\hfill
  \scale[0.33]{0.3}{1.6 \cdot 10^6}{}{1.4 \cdot 10^8}\hfill
  \scale[0.33]{0.3}{1.4 \cdot 10^6}{}{1.4 \cdot 10^8}%
  \caption{Diagonal entries of final control $q^{\rm{fin}}$ of Example 2 on a $64 \times 64$ mesh}
  \label{fig:EX_2_Control}
\end{figure}

\newpage
\section{Summary of Further Project Results and Outlook}\label{sec:outlook}
The article summarized some results obtained within the project
,,Optimizing Fracture Propagation using a Phase-Field Approach''
concerning existence and first order optimality conditions for 
control in the coefficients of a variational inequality.
Improved results and detailed proofs for these optimality conditions will be 
subject of a forthcoming publication.
Further, some
numerical results for a related coefficient control problem of 
phase-field fracture are provided.
The project analyzed in detail the control of such phase-field
fracture problems by the applied forces and the convergence in the
regularization limit in~\cite{NeitzelWickWollner:2017,NeiWiWo19}.
These results where enabled by a fundamental result on higher
integrability of solutions to elliptic systems by~\cite{HallerDintelmannMeinlschmidtWollner:2018}.
The analysis of such phase-field control problems could be extended to second order sufficient conditions~\cite{HehlNeitzel2022}, and finite element error estimates
where obtained in~\cite{MohammadiWollner:2021} for a linearized fracture control problem.  Analysis of local quadratic convergence of the SQP method for regularized fracture with control on a Neumann boundary is subject of a forthcoming publication.

Further results on optimality conditions for control of variational inequalities have been obtained in~\cite{NeitzelWachsmuth}, including
state-constraints and control in the right-hand-side, as well as for coefficient control problems in~\cite{SimonWollner:2023}.
A posteriori~\cite{EndtmayerLangerNeitzelWickWollner:2019} and a priori~\cite{HirnWollner:2022} finite element error analysis for non-smooth control problems of equations with $p$~structure could be carried out within the project. 
The project was complemented by developments of algorithms for the
control of phase-field fracture  in~\cite{Khimin2022,KhiSteiWi22_JCP} and of
Lagrange multiplier methods~\cite{KarlNeitzelWachsmuth} for nonlinear
elliptic state-constrained problems.

\section*{Acknowledgments}
Funded by the Deutsche Forschungsgemeinschaft (DFG) --
Projektnummer 314067056 within SPP 1962


\end{document}